\documentclass{article}
\usepackage{epstopdf}

\usepackage{geometry}
\usepackage{subcaption}
\usepackage{mathrsfs}
\usepackage{amssymb,newlfont,amsmath,amsthm}
\usepackage{tikz, tikz-3dplot}
\usepackage{hyperref}
\usepackage{url}
\usepackage{longtable}
\usepackage{esint}
\usepackage{xcolor}
\usepackage{tikz, pgf}
\usepackage{stmaryrd}
\usepackage{wasysym}
\SetSymbolFont{stmry}{bold}{U}{stmry}{m}{n} 
\usepackage{bm}
\usepackage{parskip}
\usetikzlibrary{arrows}
\usetikzlibrary{decorations}
\usetikzlibrary{decorations.pathmorphing}
\usetikzlibrary{calc}
\usepackage{multirow}
\usepackage{booktabs}
\usepackage[protrusion=true,expansion=true]{microtype}
\usepackage{graphicx}

\newtheorem{theorem}{Theorem}[section] 
\newtheorem{corollary}[theorem]{Corollary} 

\newtheorem{lemma}[theorem]{Lemma}

\newtheorem{remark}[theorem]{Remark}

\title{A residual-based finite element surrogate solver \\ for elliptic partial differential equations}

\author{Kyoungjin Jung,\thanks{Department of Mathematics, Ajou University, Suwon, Republic of Korea. Email: \tt{fned49@ajou.ac.kr}}
~Jae Yong Lee\thanks{Department of AI, Chung-Ang University, Seoul, Republic of Korea. Email: \tt{jaeyong@cau.ac.kr}}
~and ~Dongwook Shin\thanks{Department of Mathematics, Ajou University, Suwon, Republic of Korea. Email: \tt{dws@ajou.ac.kr}}}

\date{}

\begin{document}

\maketitle

\begin{abstract}
    We propose a residual-based finite element surrogate solver for elliptic partial differential equations.
    The method combines convolutional neural networks with classical finite element discretization in a data-free setting, where the loss function is defined directly from the finite element residual.
    This enables the approximation of the solution operator without requiring paired input-output data.
    A key feature of the proposed approach is that it can be analyzed using standard finite element theory under mesh refinement.
    We establish a relationship between the training loss and the error in the $H^1$-seminorm, and derive training criteria that ensure optimal convergence rates.
    To improve efficiency, we introduce a decomposition strategy that separates the contributions of different input components.
    This allows the model to learn simpler sub-operators.
    Numerical experiments demonstrate that the proposed method achieves stable convergence under grid refinement, remains robust on complex geometries and oscillatory solutions, and extends naturally to nonlinear equations.
\end{abstract}



\section{Introduction}

Solving partial differential equations (PDEs) is a fundamental task in computational science, with applications ranging from physics and engineering to climate modeling and biology. Classical numerical solvers, such as the finite difference method (FDM) and the finite element method (FEM), provide accurate solutions but often require significant computational cost when solving parameterized families of PDEs across multiple input configurations. This limitation becomes particularly severe on fine grids, where repeated solutions of large linear systems or the storage of dense inverse operators are prohibitive.

Operator learning~\cite{LU2022114778,JMLR:v24:21-1524} has gained attention as an alternative approach that aims to directly learn mappings from input functions, such as source terms and boundary data, to PDE solutions. Despite this progress, existing neural operators exhibit several critical limitations. First, most operators are trained in a supervised setting and rely on large datasets of input–output pairs, which are expensive to generate using classical solvers~\cite{kovachki2024data}. This dependence severely limits scalability, especially in applications where ground-truth solutions are unavailable or prohibitively expensive to compute. Second, generalization to diverse boundary conditions remains a significant challenge~\cite{wang2024beno}. Many existing models are designed under fixed or simplified boundary types—such as periodic or purely Dirichlet boundary conditions—and often fail to adapt to new configurations, particularly when Neumann or mixed boundary conditions are introduced. Third, current approaches typically treat PDEs as black-box mappings without exploiting known mathematical structures such as linearity or decomposability. As a result, models may overfit specific input patterns and lack the interpretability and modularity needed for more complex scientific applications. 

Recently, convolution-based neural operator architectures have been proposed to improve the scalability and expressivity of operator learning models~\cite{MR3800689,MR3957452}. 
In particular, the convolutional neural operator~\cite{raonic2023convolutional} introduces structure-preserving modifications to Convolutional Neural Networks (CNNs) to enforce continuous-discrete equivalence and reduce aliasing errors.
This approach provides a principled operator-level adaptation of U-Net architectures~\cite{ronneberger2015u}. 
This line of work addresses important representation-theoretic challenges in operator learning, especially those related to discretization dependence and aliasing effects in convolutional architectures.
However, these advances primarily focus on how operators are represented and learned across resolutions, rather than on the incorporation of physical laws and numerical discretizations into the training process.
In particular, the governing PDE structure typically enters such models only implicitly, and the training procedure remains largely decoupled from classical numerical schemes.
As a result, these models do not explicitly exploit the consistency, locality, and stability properties inherent in numerical methods.
In many scientific and engineering applications, the governing equations are explicitly known, while the paired solution data may be expensive or impractical to obtain.
This observation motivates operator learning frameworks that directly integrate numerical discretizations and physics-based constraints into the learning process.
Such frameworks enable training driven by the PDE itself rather than by precomputed solution pairs.

These considerations motivate the development of operator learning frameworks that explicitly integrate physical laws and numerical discretizations into the learning process.
Our approach is closely related to Finite Element Operator Network (FEONet)~\cite{MR4888707}, which learns solution operators without paired input-output data through a finite element residual formulation.
We extend this framework to a convolutional surrogate solver by introducing an image-based finite element representation.
Instead of reshaping a finite element input vector into an image, the proposed method embeds the computational domain into a fixed image grid and constructs the finite element mesh from the corresponding nodes.
This design preserves the spatial structure of the physical domain and facilitates the treatment of complex geometries within a convolutional architecture.
We refer to the resulting method as the finite element convolutional operator network (FE-CON).
For completeness, we also present a finite difference formulation (FD-CON) in Appendix~C, which follows the same residual-based construction.
The main contributions of this paper are as follows:
\begin{itemize}
    \item \textbf{Image-based finite element convolutional operator framework.} We extend the FEONet framework to a residual-based convolutional surrogate solver by introducing an image-based finite element representation. The proposed formulation directly incorporates finite element discretization into the learning objective. This formulation enables operator learning without paired solution data and establishes a direct connection between neural approximation and finite element analysis.
    \item \textbf{Error analysis and training criterion for optimal convergence.} We derive error estimates in the $H^1$-seminorm and establish a relationship between the decay of the training loss and the approximation error. Based on this result, we propose a training criterion that guarantees optimal convergence rates under mesh refinement.
    \item \textbf{Decomposition strategy for efficient operator learning.} We introduce a decomposition strategy that splits the original problem into simpler subproblems. This approach improves training efficiency and generalization, particularly when only a limited number of training samples are available.
    \item \textbf{Learning mixed boundary conditions with fully varying data.} We consider the Poisson equation with mixed (Dirichlet--Neumann) boundary conditions, where the source term and both boundary conditions vary independently.
    This setting differs from prior works, which typically vary only one component while keeping the others fixed.
    \item \textbf{Robust performance on fine grids and complex settings.} We demonstrate that the proposed method trains reliably on fine grids and produces accurate solutions under grid refinement. Additional experiments show the relation between generalization error and training sample size, and illustrate the performance on complex settings, such as complex geometries, the Helmholtz equation, and a nonlinear equation.
\end{itemize}

The remainder of this paper is organized as follows.
The next section provides a brief introduction to classical numerical schemes and related work. 
Section~\ref{sec:methodology} describes the proposed method and Section~\ref{sec:analysis} presents the error analysis.
Section~\ref{sec:numerical_exp} provides experimental results demonstrating the accuracy, training efficiency, and scalability of our approach. 
We conclude in Section~\ref{sec:conclusion} with a discussion of future directions.

\section{Preliminaries}\label{sec:preliminaries}
In this section, we review the related work and background necessary for this work.
We first describe the model problem and its finite element discretization.
Next, we present an operator perspective of the resulting discrete system.
Finally, we review various machine-learning-based approaches for learning PDE operators.

\subsection{Model problem and finite element discretization} \label{sec:FEM}
Let us consider the following Poisson equation with mixed boundary conditions:
\begin{subequations}\label{Poisson_problem}
\begin{align}
	- \Delta u &= f \ \ \quad \textrm{in} \ \ \Omega \label{Poisson_f}, \\ 
    u &= g_D \quad \textrm{on} \ \partial\Omega_D \label{Poisson_Diri}, \\
	\nabla u \cdot \boldsymbol{n} &= g_N \quad \textrm{on} \ \partial\Omega_N, \label{Poisson_Neu}
\end{align}
\end{subequations}
where $\Omega \subset \mathbb{R}^d \ (d = 1,\ 2,\ 3)$ is a bounded Lipschitz domain with polygonal boundary $\partial\Omega$, $\partial\Omega_D$ is the Dirichlet boundary, and $\partial\Omega_N$ is the Neumann boundary.
Here, $\partial\Omega_D$ is a closed subset of $\partial\Omega$ with positive measure and $\partial\Omega_N = \partial\Omega \setminus\partial\Omega_D$.
For simplicity, we focus on the two-dimensional case ($d = 2$) throughout the remainder of this paper.

We summarize the finite element formulation for the problem \eqref{Poisson_problem}.
Let $u_D$ be an extension of $g_D$ to the entire domain such that $u_D = g_D$ on $\partial\Omega_D$.
For $f \in L^2(\Omega)$, $u_D \in H^1(\Omega)$, and $g_N \in L^2(\partial\Omega_N)$, there exists a weak solution $u\in H^1(\Omega)$ by the Lax–Milgram lemma.
In order to deal with the inhomogeneous Dirichlet conditions \eqref{Poisson_Diri}, we introduce the decomposition $u = w + u_D$, where $w \in H^1_D(\Omega)$ and
$$H^1_D(\Omega) := \{ v \in H^1(\Omega) \mid v = 0 \text{ on } \partial\Omega_D \}.$$
The weak formulation reads: find $w \in H^1_D(\Omega)$ such that
\begin{equation}\label{weak_form}
\int_{\Omega} \nabla w \cdot \nabla v \, dx = \int_{\Omega} f v \, dx + \int_{\partial\Omega_N} g_N v \, ds - \int_{\Omega} \nabla u_D \cdot \nabla v \, dx,
\end{equation}
for all $v\in H^1_D(\Omega)$.
This decomposition will also play a key role in the formulation of the proposed method.

For the numerical approximation, we consider finite element spaces \( S \subset H^1(\Omega) \) and \( S_D \subset S \cap H^1_D(\Omega) \).
For a given \( U_D \in S \) which approximates \( u_D \) on \( \partial\Omega_D \), the discrete problem reads: Find \( W \in S_D \) such that
\begin{equation}
\int_{\Omega} \nabla W \cdot \nabla V \, dx = \int_{\Omega} f V
\, dx + \int_{\partial\Omega_N} g_N V \, ds - \int_{\Omega} \nabla U_D \cdot \nabla V \, dx, \label{eq:discrete}
\end{equation}
for all $V \in S_D$.
Let $\{\eta_i\}_{i=1}^N$ denote the node points on a given mesh of $\Omega$, and let $\{\psi_i\}_{i=1}^N$ be the corresponding nodal basis of $S$ such that $\psi_i(\eta_j) = \delta_{ij}$.
Let  $\{\psi_{i_k}\}_{k=1}^M$ denote the basis of $S_D$, where $I = \{i_1, ..., i_M\} \subseteq \{1,...,N\}$ is an index set of cardinality $M \le N-2$.
Then, the discrete problem \eqref{eq:discrete} is equivalent to the linear system
\begin{equation} \label{FEsystem}
	A\boldsymbol{w} = \boldsymbol b
\end{equation}
where the stiffness matrix $A\in\mathbb{R}^{M\times M}$ and the load vector $\boldsymbol{b}\in\mathbb{R}^M$ are defined by
\begin{subequations} \label{FEM_matrix_final}
\begin{align}
A_{\ell m} &= \int_{\Omega} \nabla \psi_\ell \cdot \nabla \psi_m \, dx, \label{matrix_A} \\
b_\ell &= \int_{\Omega} f \psi_\ell \, dx + \int_{\partial\Omega_N} g_N \psi_\ell \, ds - \sum_{k=1}^{N} U_k \int_{\Omega} \nabla \psi_\ell \cdot \nabla \psi_k \, dx. \label{FE_system_b}
\end{align}
\end{subequations}
The matrix $A$ is sparse, symmetric and positive definite, and therefore the system admits a unique solution $\boldsymbol{w} \in \mathbb{R}^M$. Then the finite element solution is given by
\begin{equation}\label{FEsolution_rep}
U = W + U_D = \sum_{k \in I} w_k \psi_k + \sum_{i=1}^{N} U_i \psi_i.
\end{equation}

The discrete problem can be interpreted as the application of the inverse operator $A^{-1}$ to the load vector $\boldsymbol{b}$.
Here, the vector $\boldsymbol{b}$ is determined by the PDE data, such as the source term, the boundary conditions, and the geometric information of the domain.
Thus, each set of PDE data defines a different load vector and a corresponding discrete solution.
In this work, we aim to approximate this mapping without explicitly solving the linear system for each input.

\subsection{Operator perspective of the discrete problem} \label{sec:operator_perspective}
The discrete problem introduced in Section~\ref{sec:FEM} defines a mapping from the input PDE data, such as the source term, the boundary conditions, and the geometric information of the domain, to the corresponding solution.
In the finite element setting, this mapping associates the load vector with the discrete solution obtained from the variational formulation.
From this perspective, the problem is to construct a solution operator that maps the given data to the corresponding solution.

A more general characterization of the solution is based on a residual formulation.
Let $r(u; f, g_D, g_N)$ denote the residual associated with a given problem.
Then, the solution is the function that satisfies the residual equation
$$r(u;f, g_D, g_N) = 0,$$
in a suitable weak sense.
For example, in the linear case considered in Section~\ref{sec:FEM}, the residual corresponds to the variational formulation, where the solution satisfies
$$r(w; f, g_D, g_N) (v) = \int_\Omega f v\ dx + \int_{\partial\Omega_N} g_N v \, ds - \int_{\Omega} \nabla u_D \cdot \nabla v \, dx - \int_\Omega \nabla w\cdot \nabla v\ dx = 0,$$
for all test functions $v\in H^1_D(\Omega)$.
In the discrete setting, this condition leads to the linear system $A\boldsymbol{w} = \boldsymbol{b}$.
Equivalently, if a discrete function $w_h$ has the coefficient vector $\boldsymbol{w}$, the discrete residual can be written as $\boldsymbol{r}(\boldsymbol{w}; A, \boldsymbol{b}) = \boldsymbol{b} - A\boldsymbol{w}$.
In this case, the mapping from the input data to the solution can be expressed through the inverse of the stiffness matrix.
The residual formulation also applies to nonlinear problems.
In nonlinear elliptic problems, the residual includes nonlinear terms in the solution.
The solution is then defined as the function that satisfies the nonlinear residual equation.
Classical numerical methods use iterative solvers, such as Newton-type methods.
These methods update the solution until the residual becomes sufficiently small.

In this work, we use the residual formulation to construct a surrogate solver.
We do not solve the linear or nonlinear system iteratively for each input.
Instead, we approximate the mapping from the input data to the solution by enforcing the residual equation.
This perspective provides a unified framework for both linear and nonlinear problems and serves as the foundation of the proposed method in Section~\ref{sec:methodology}.

\subsection{Related work}\label{sec:related_work}
Recent advances in physics-informed machine learning \cite{karniadakis2021physics} have led to a growing interest in solving PDEs using neural networks. Representative examples include physics-informed neural networks (PINNs)~\cite{MR3881695}, which embed physical laws into the loss function. 
Building on this idea, an increasing body of work has focused on learning mappings between infinite-dimensional function spaces—known as neural operators. 
Major neural operator architectures include the Deep Operator Network (DeepONet) \cite{lu2021learning, LU2022114778, lee2023hyperdeeponet} and the Fourier Neural Operator (FNO) \cite{li2021fourier}, together with extensions based on meta-learning~\cite{pmlr-v235-cho24b,cho2023hypernetworkbased}, graph neural networks~\cite{brandstetter2022message,boussif2022magnet,CHO2026114430}, and transformers~\cite{pmlr-v202-hao23c,wang2025cvit,cao2021choose}.

CNN-based models have also adopted U-Net architectures to enable robust and accurate operator learning in image-based domains, along with theoretical results on their approximation capacity~\cite{raonic2023convolutional}. In addition, hybrid approaches that integrate classical numerical solvers, such as FDM~\cite{fdm-nn}, with deep learning have been explored to improve data efficiency and utilize domain-specific inductive biases. Despite these successes, most neural operator frameworks are trained in a supervised manner and require large collections of paired input–output data, which are often generated via computationally expensive numerical solvers. This reliance on labeled data presents a critical bottleneck, particularly for large-scale or parametric PDEs ~\cite{de2022cost}.

To mitigate this limitation, data-free or physics-informed neural operators have recently gained attention~\cite{Goswami2023}. For example, Zhu et al.~\cite{MR3957452} proposed a CNN-based method that incorporates finite difference discretizations and learns operators directly from PDE residuals. The Physics-Informed DeepONet (PI-DeepONet)~\cite{wang2021learning} extended DeepONet to an unsupervised setting by minimizing residuals instead of relying on labeled data. Similarly, the Physics-Informed Neural Operator (PINO)~\cite{PINO} introduced a loss formulation that enables FNO training with limited or no supervision. 
Several studies have also investigated the treatment of boundary conditions in neural operator frameworks~\cite{wang2024beno,sukumar2022exact,MR4888707}. While progress has been made in handling Dirichlet or periodic boundary conditions through network architectures and loss design, a general and scalable approach that can handle mixed or variable boundary conditions across resolutions remains underexplored. Our work addresses this gap by proposing a decomposition-based formulation that naturally accommodates diverse boundary configurations while maintaining fully data-free training.

More recently, FEONet~\cite{MR4888707}, proposed by Lee et al., leveraged a finite element formulation to learn solution operators without labeled data, achieving strong performance across a variety of PDE benchmarks.
The proposed approach extends the FEONet framework.
In the CNN-based implementation of FEONet, the input vector is reshaped into an image to match the CNN architectures.
However, this reshaping may distort the spatial structure, particularly for non-rectangular or unstructured meshes, where neighboring pixels can correspond to distant physical locations.
In addition, the input dimension must match the number of image pixels, which imposes constraints for complex geometries.
In contrast, the proposed method maps the computational domain onto a fixed image grid and constructs the mesh from the generated nodes.
Since the finite element data are embedded directly into this image-based representation, the input is naturally defined on the same grid and its dimension matches the number of image pixels.
This design preserves spatial consistency and facilitates the application to complex geometries (see Figure~\ref{fig:domain_to_image}).

\begin{figure}
    \centering
    \includegraphics[height=4.9cm]{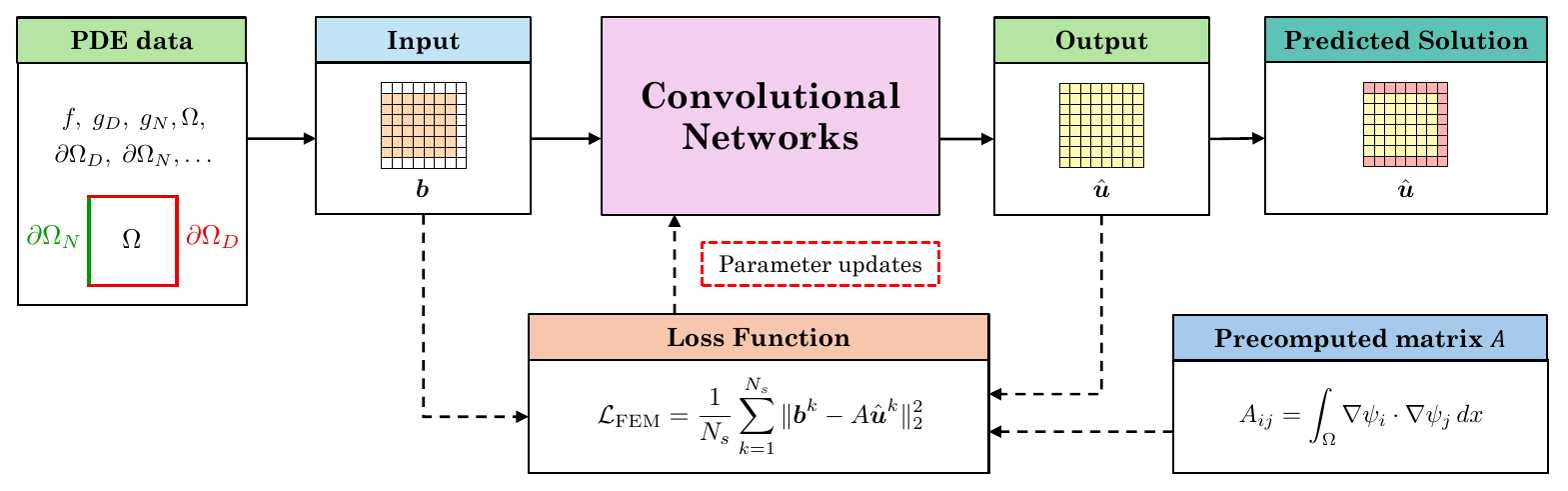}
    \caption{Overview of the proposed residual-based convolutional method.}
    \label{fig:framework}
\end{figure}

\section{Methodology}\label{sec:methodology}
In this section, we describe the proposed residual-based convolutional method.
The method constructs a surrogate solver that approximates the mapping from input data to the corresponding solution without explicitly solving the underlying linear or nonlinear system.
An overview of the proposed method is shown in Figure \ref{fig:framework}.

\subsection{Image-based representation and problem decomposition}\label{sec:decomposition}

We first represent the computational domain on a fixed image grid and construct the corresponding mesh from this representation.
For illustration, we consider a computational domain $\Omega \subset \mathbb{R}^2$ embedded in a square image domain, and we discretize the image domain into a uniform square grid of size $N \times N$. 
Figure~\ref{fig:domain_to_image} shows the case where $\Omega$ is a square domain with a square hole.
In classical FEM, each node corresponds to a mesh vertex. 
In contrast, we adopt an image-based representation.
Here, each node is represented as a pixel in an image, and the domain is embedded into this image by assigning values to the corresponding pixels.

More precisely, we first specify the size of the input image and then map the computational domain onto this image.
Then, blue nodes are placed at the pixels corresponding to the interior and boundary of the domain, and white nodes are placed at the pixels corresponding to regions outside the domain, as illustrated in Figure~\ref{fig:image-nodes}.
Finally, the mesh is constructed based on the positions of the blue nodes (see Figure~\ref{fig:mesh}).
We note that, although node alignment on a uniform grid simplifies the implementation of the finite element method, such an alignment is not essential for the proposed approach.
Node locations can be adjusted as long as the underlying geometric information is preserved.
While this adjustment may introduce apparent distortion in the image representation, it does not affect training since the loss function is evaluated at the true spatial locations of the nodes.
The solution can also be visualized accurately by reconstructing it using the physical coordinates.
All nodes are numbered in the order of the image pixels, regardless of whether they are blue or white. 

The load vector $\boldsymbol{b}$ associated with \eqref{FE_system_b} is embedded into the full image grid so that its dimension matches the number of pixels in the input image.
Only the entries associated with the degrees of freedom in the FE formulation are evaluated through numerical integration using the true spatial coordinates, while the remaining entries are set to zero.
Moreover, no physical or spatial distortion is introduced because all computations in the loss function are carried out using the true spatial coordinates of the nodes.
A masking image can be used to enforce zero values at the corresponding pixels in the output as well.
This approach allows us to embed the spatial structure of the PDE into a convolutional architecture, enabling the use of CNNs to approximate differential operators efficiently. 

\begin{figure}[t!]
\centering

\begin{subfigure}[t]{0.32\textwidth}
\centering
\begin{tikzpicture}[scale=0.65, transform shape]

\draw[white, line width = 1pt] (0,0)--(5,0)--(5,5)--(0,5)--(0,0)--(1,0);

\draw[blue, thick, fill=blue!10] (0.25,0.25)--(0.25,4.75)--(4.75,4.75)--(4.75,0.25)--(0.25,0.25)--(0.25,1);
\draw[blue, thick, fill=white] (1.75,1.75)--(1.75,3.25)--(3.25,3.25)--(3.25,1.75)--(1.75,1.75)--(1.75,2);

\end{tikzpicture}
\caption{Domain}
\label{fig:domain}
\end{subfigure}
\hfill
\begin{subfigure}[t]{0.32\textwidth}
\centering
\begin{tikzpicture}[scale=0.65, transform shape]

\foreach \x in {1,...,9} {
  \draw[black, line width=1pt] (0.5*\x,0) -- (0.5*\x,5);
}
\foreach \y in {1,...,9} {
  \draw[black, line width=1pt] (0,0.5*\y) -- (5,0.5*\y);
}

\draw[black, line width = 1pt] (0,0)--(5,0)--(5,5)--(0,5)--(0,0)--(1,0);

\draw[blue, thick] (0.25,0.25)--(0.25,4.75)--(4.75,4.75)--(4.75,0.25)--(0.25,0.25)--(0.25,1);
\draw[blue, thick] (1.75,1.75)--(1.75,3.25)--(3.25,3.25)--(3.25,1.75)--(1.75,1.75)--(1.75,2);

\foreach \x in {0,...,3} {
  \foreach \y in {0,...,9} {
    \fill[blue!100] (0.5*\x+0.25,0.5*\y+0.25) circle (2pt);  
  }
}

\foreach \x in {6,...,9} {
  \foreach \y in {0,...,9} {
    \fill[blue!100] (0.5*\x+0.25,0.5*\y+0.25) circle (2pt);  
  }
}

\foreach \x in {4,5} {
  \foreach \y in {0,...,3} {
    \fill[blue!100] (0.5*\x+0.25,0.5*\y+0.25) circle (2pt);  
  }
}

\foreach \x in {4,5} {
  \foreach \y in {6,...,9} {
    \fill[blue!100] (0.5*\x+0.25,0.5*\y+0.25) circle (2pt);  
  }
}

\draw[blue] (2.25, 2.25) circle (2pt);
\draw[blue] (2.75, 2.25) circle (2pt);
\draw[blue] (2.25, 2.75) circle (2pt);
\draw[blue] (2.75, 2.75) circle (2pt);

\end{tikzpicture}
\caption{Nodes on an image}
\label{fig:image-nodes}
\end{subfigure}
\hfill
\begin{subfigure}[t]{0.32\textwidth}
\centering
\begin{tikzpicture}[scale=0.65, transform shape]

\foreach \x in {1,...,9} {
  \draw[black, line width=1pt] (0.5*\x,0) -- (0.5*\x,5);
}
\foreach \y in {1,...,9} {
  \draw[black, line width=1pt] (0,0.5*\y) -- (5,0.5*\y);
}

\draw[black, line width = 1pt] (0,0)--(5,0)--(5,5)--(0,5)--(0,0)--(1,0);

\foreach \x in {0,1,2,3} {
  \draw[blue, thick] (0.5*\x+0.25, 0.25) -- (0.5*\x+0.25,4.75);
  \draw[blue, thick] (0.25, 0.5*\x+0.25) -- (4.75, 0.5*\x+0.25);
  \draw[blue, thick] (4.75 - 0.5*\x, 0.25) -- (4.75 - 0.5*\x,4.75);
  \draw[blue, thick] (0.25, 4.75 - 0.5*\x) -- (4.75, 4.75 - 0.5*\x);
}

\foreach \x in {1,2} {
  \draw[blue, thick] (0.5*\x+1.75, 0.25) -- (0.5*\x+1.75,1.75);
  \draw[blue, thick] (0.5*\x+1.75, 3.25) -- (0.5*\x+1.75,4.75);
  \draw[blue, thick] (0.25, 0.5*\x+1.75) -- (1.75, 0.5*\x+1.75);
  \draw[blue, thick] (3.25, 0.5*\x+1.75) -- (4.75, 0.5*\x+1.75);
}

\foreach \x in {0,...,3} {
  \foreach \y in {0,...,9} {
    \fill[blue!100] (0.5*\x+0.25,0.5*\y+0.25) circle (2pt);  
  }
}

\foreach \x in {6,...,9} {
  \foreach \y in {0,...,9} {
    \fill[blue!100] (0.5*\x+0.25,0.5*\y+0.25) circle (2pt);  
  }
}

\foreach \x in {4,5} {
  \foreach \y in {0,...,3} {
    \fill[blue!100] (0.5*\x+0.25,0.5*\y+0.25) circle (2pt);  
  }
}

\foreach \x in {4,5} {
  \foreach \y in {6,...,9} {
    \fill[blue!100] (0.5*\x+0.25,0.5*\y+0.25) circle (2pt);  
  }
}

\draw[blue] (2.25, 2.25) circle (2pt);
\draw[blue] (2.75, 2.25) circle (2pt);
\draw[blue] (2.25, 2.75) circle (2pt);
\draw[blue] (2.75, 2.75) circle (2pt);

\end{tikzpicture}
\caption{Mesh}
\label{fig:mesh}
\end{subfigure}

\caption{Illustration of the mesh generation process based on image-domain mapping.}
\label{fig:domain_to_image}
\end{figure}

To improve efficiency and flexibility, we decompose the model problem \eqref{Poisson_problem} into two subproblems.
The first subproblem is defined by
\begin{subequations}\label{Subproblem1}
\begin{align}
	- \Delta u_{1} &= f \ \ \quad \textrm{in} \ \ \Omega \label{Problem1_f} \\
    u_{1} &= 0 \, \ \ \quad \textrm{on} \ \partial\Omega_D \label{Problem1_Diri} \\
	\nabla u_{1} \cdot \boldsymbol{n} &= g_N \quad \textrm{on} \ \partial\Omega_N \label{Problem1_Neu} 
\end{align}
\end{subequations}
and the second subproblem is defined by
\begin{subequations}\label{Subproblem2}
\begin{align}
	- \Delta u_{2} &= 0 \ \ \quad \textrm{in} \ \ \Omega \label{Problem2_f} \\
	u_{2} &= g_D \quad \textrm{on} \ \partial\Omega_D \label{Problem2_Diri} \\
    \nabla u_{2} \cdot \boldsymbol{n} &= 0 \ \ \quad \textrm{on} \ \partial\Omega_N \label{Problem2_Neu} 
\end{align}
\end{subequations}
By linearity of the problem, the solution to the original problem \eqref{Poisson_problem} can be written as
\begin{equation} \label{decomposition}
u = u_1 + u_2.
\end{equation}
This decomposition separates the contributions of the source term and the Neumann boundary condition from those of the Dirichlet boundary condition.
Each subproblem can be treated independently and its solutions are combined to obtain the final solution.
This decomposition simplifies the learning task and improves training efficiency.

\subsection{Finite Element Convolutional Operator Network}\label{subsec:FEUNet}
We now describe the finite element convolutional operator network (FE-CON).
The method approximates the mapping from the load vector to the corresponding finite element solution by using a convolutional neural network and a residual-based loss function.

The FE-CON model takes a single-channel image corresponding to the load vector $\boldsymbol{b}$ as input.
The input is defined on the full image grid introduced in Section~\ref{sec:decomposition}, and its dimension matches the number of pixels.
The output is also a single-channel image that represents the predicted solution, and it shares the same spatial resolution as the input.
After prediction, the Dirichlet boundary values are enforced by overwriting the corresponding pixels with the prescribed boundary data.

Based on the residual formulation in Section~\ref{sec:operator_perspective}, we define the following loss function:
\begin{equation}\label{fe_loss}
\mathcal{L}_{\text{FEM}} = \frac{1}{N_s} \sum_{k = 1}^{N_s} \| \boldsymbol{b}^k - A  \hat{\boldsymbol{u}}^k\|_2^2,
\end{equation}
where $N_s$ is the number of training samples, $\boldsymbol{b}^k$ is the $k-$th input, and $\hat{\boldsymbol{u}}^k$ is the corresponding output.
This loss measures the residual of the discrete system and enforces consistency with the finite element formulation.

One way to justify the decomposition \eqref{Subproblem1}--\eqref{Subproblem2} is through the linear system of equations \eqref{FEsystem} arising from the FEM formulation.
The right-hand side vector $\boldsymbol{b}$ in the linear system can be decomposed as $\boldsymbol{b} = \boldsymbol{b}_1 + \boldsymbol{b}_2$, where two vectors are defined as follows:
\begin{subequations}
\begin{align}
(\boldsymbol{b}_1)_j &= \int_{\Omega} f \psi_j \, dx + \int_{\partial\Omega_N} g_N \psi_j \, ds, \label{b1} \\
(\boldsymbol{b}_2)_j &= - \sum_{k=1}^{N} U_k \int_{\Omega} \nabla \psi_j \cdot \nabla \psi_k \, dx. \label{b2}
\end{align}    
\end{subequations}
Here, $\boldsymbol{b}_1$ arises naturally from the variational formulation of FEM, while $\boldsymbol{b}_2$ results from the decomposition 
$u = w + u_D$ introduced in Section~\ref{sec:preliminaries}.
The solution to the linear system in equation \eqref{FEsystem} can then be expressed as
$$\boldsymbol{w} = A^{-1} \boldsymbol{b} = A^{-1} (\boldsymbol{b}_1 + \boldsymbol{b}_2) = A^{-1} \boldsymbol{b}_1 + A^{-1} \boldsymbol{b}_2.$$
When expressed in \eqref{FEsolution_rep}, $A^{-1} \boldsymbol{b}_1$ corresponds to the solution of Subproblem 1, and $A^{-1} \boldsymbol{b}_2+ U_D$ corresponds to the solution of Subproblem 2.
For the subproblems, the term $\boldsymbol{b}$ in the loss function \eqref{fe_loss} is replaced with $\boldsymbol{b}_1$ or $\boldsymbol{b}_2$, depending on the subproblem.

The FE-CON model can be applied either to the original problem or to each subproblem in the decomposition.
In the latter case, the network is trained separately for $\boldsymbol{b}_1$ and $\boldsymbol{b}_2$, and the final solution is obtained by combining the predicted solutions according to the decomposition $u = u_1 + u_2$.
The decomposition-based approach reduces the complexity of the learning task by splitting the problem into simpler components.
However, it requires training two separate models.
For this reason, we also consider a single model that directly approximates the solution to the original problem.

\section{Error Analysis} \label{sec:analysis}
In this section, we present an error analysis of the proposed method and establish a relation between the residual-based loss and the approximation error.
To simplify the error analysis, we consider the Poisson equation on the unit square domain $\Omega = (0,1)^2$ with homogeneous Dirichlet boundary conditions, i.e., $g_D = 0$ and $\partial\Omega_D = \partial\Omega$.
We discretize $\Omega$ by a uniform $N\times N$ grid and use the corresponding finite element mesh with mesh size $h = 1/(N-1)$.

The linear system takes the same form as \eqref{FEsystem}, but $\boldsymbol{b}$ is given by
\begin{equation*}
    b_j = \int_\Omega f \psi_j\ dx.
\end{equation*}
Here, $U = W$ and $U \in S\cap H^1_D(\Omega)$.

Let $M$ denote the global mass matrix defined by
\begin{equation} \label{mass_matrix}
    M_{ij} = \int_\Omega \psi_i \psi_j\ dx.
\end{equation}
Since the basis functions $\{\psi_i\}$ are linearly independent, the global mass matrix $M$ is symmetric positive definite.
Let $n_h$ denote the number of finite element degrees of freedom.
If $v_h \in S \cap H^1_D(\Omega)$, then $v_h= \sum_{i=1}^{n_h} v_i \psi_i$.
Consequently, $\nabla v_h = \sum_{i=1}^{n_h} v_i \nabla\psi_i$.
Let $\boldsymbol{v} = (v_1, \ldots, v_{n_h}) \in \mathbb{R}^{n_h}$ denote the coefficient vector.
Then, from \eqref{matrix_A} and \eqref{mass_matrix}, it follows that
\begin{equation} \label{norm_w_matrix}
    \|\nabla v_h\|^2_{L^2(\Omega)} = \boldsymbol{v}^T A \boldsymbol{v} \qquad \text{and} \qquad \|v_h\|^2_{L^2(\Omega)} = \boldsymbol{v}^T M \boldsymbol{v}.
\end{equation}

\begin{lemma}
    Let $v_h = \sum_{i=1}^{n_h} v_i \psi_i \in S \cap H^1_D(\Omega)$  and $\boldsymbol{v} = (v_1, \ldots, v_{n_h}) \in \mathbb{R}^{n_h}$. If $\lambda_{\text{min}}$ is the minimum eigenvalue of the global mass matrix $M$, then
    \begin{equation} \label{l2L2}
        \|\boldsymbol{v}\|_2^2 \le \frac{1}{\lambda_{\text{min}}} \|v_h\|_{L^2(\Omega)}^2.
    \end{equation}
\end{lemma}
\begin{proof}
    \begin{align*}
        \|\boldsymbol{v}\|_2^2 \le \frac{1}{\lambda_{\text{min}}} \boldsymbol{v}^T M \boldsymbol{v} = \frac{1}{\lambda_{\text{min}}} \|v_h\|_{L^2(\Omega)}^2.
    \end{align*}
\end{proof}

The following theorem establishes an error estimate.

\begin{theorem}
    Let $u\in H^2(\Omega)\cap H^1_D(\Omega)$ be the weak solution and $\hat{u} \in S\cap H^1_D(\Omega)$ be the FE-CON prediction. Then
    \begin{equation}
        \|\nabla u - \nabla \hat{u}\|_{L^2(\Omega)}^2 \le C_1 h^2 \|u\|^2_{H^2(\Omega)} + C_2 h^{-2} \|\boldsymbol{b} - A \hat{\boldsymbol{u}}\|_2^2.
    \end{equation}
\end{theorem}
\begin{proof}
    Let $u_h^* \in S\cap H^1_D(\Omega)$ be the finite element solution. 
    The triangle inequality and a priori error estimate of FEM yield
    \begin{equation} \label{error_bound}
        \|\nabla u - \nabla \hat{u}\|_{L^2(\Omega)}^2 \le C h^2 \|u\|_{H^2(\Omega)}^2 + 2\|\nabla u_h^* - \nabla \hat{u}\|_{L^2(\Omega)}^2. 
    \end{equation}
    By Cauchy-Schwarz inequality, the discrete Poincar\'e inequality, \eqref{norm_w_matrix}, and \eqref{l2L2}, we have
    \begin{align*}
        \|\nabla u_h^* - \nabla \hat{u}\|_{L^2(\Omega)}^2 &\le \|\boldsymbol{u}^*_h - \hat{\boldsymbol{u}}\|_2 \|\boldsymbol{b} - A \hat{\boldsymbol{u}}\|_2 \\
        &\le \frac{C_P}{\sqrt{\lambda_{\textrm{min}}}} \|\nabla u_h^* - \nabla\hat{u}\|_{L^2(\Omega)} \|\boldsymbol{b} - A \hat{\boldsymbol{u}}\|_2
    \end{align*}
    Since $\lambda_{\text{min}} \ge C_M h^2$, we have
    $$\|\nabla u_h^* - \nabla \hat{u}\|_{L^2(\Omega)} \le C h^{-1}\|\boldsymbol{b} - A \hat{\boldsymbol{u}}\|_2.$$
    Substituting the above estimate into \eqref{error_bound} yields the desired result. This completes the proof.
\end{proof}

\begin{corollary} \label{cor_train}
Let $u^k\in H^2(\Omega)\cap H^1_D(\Omega)$ be the weak solution for $f^k$ in the homogeneous Dirichlet boundary setting considered above. 
Let $\hat{u}^k \in S\cap H^1_D(\Omega)$ be the FE-CON prediction corresponding to $\boldsymbol{b}^k$. 
Then
\begin{equation}
    \frac{1}{N_s} \sum_{k=1}^{N_s} \|\nabla u^k - \nabla \hat{u}^k\|_{L^2(\Omega)}^2 \le C_1 h^2 \frac{1}{N_s} \sum_{k=1}^{N_s} \|u^k\|_{H^2(\Omega)}^2 + C_2 h^{-2} \mathcal{L}_{\text{FEM}}.
\end{equation}
\end{corollary}

When the vectors $\{\boldsymbol{b}^k\}_{k=1}^{N_s}$ are used as the training inputs, the left-hand side of Corollary~\ref{cor_train} represents the average of the squared $H^1$-seminorm error over the training set.
Thus, if the loss $\mathcal{L}_{\text{FEM}}$ is sufficiently small in the sense specified below, the average of squared error is bounded by $O(h^2)$.
This result shows that the training loss directly controls the approximation error and provides a guideline for achieving optimal convergence.

\begin{remark}\label{rem:loss_scaling}
     To achieve the optimal rate of convergence over the training set, $\mathcal{L}_{\text{FEM}}$ must be bounded by $O(h^4)$. 
     This suggests a strategy for training. If $\mathcal{L}_{\text{FEM}} = L_c$ when $h = h_c$, then FE-CON should be trained until $\mathcal{L}_{\text{FEM}} \le L_c / 16$ when $h = h_c / 2$.
\end{remark}

\begin{remark}\label{rem:high_order}
The above analysis suggests that the training loss must decrease at a rate consistent with the finite element approximation error in order to achieve optimal convergence.
This principle extends naturally to higher-order finite element spaces, where the approximation error scales as $O(h^{2\ell})$ for a $P_\ell$ finite element approximation.
In particular, this implies that the loss should decrease at a rate of $O(h^{2\ell+2})$, based on the balance between the approximation error and the residual term in Corollary~\ref{cor_train}.
\end{remark}

\begin{remark}\label{rmk:hong2024}
    In \cite{hong2024error}, the error is decomposed into the FEM error, approximation error, and generalization error.
    In particular, Theorem 4.10 bounds the approximation and generalization errors by $\kappa(A)^{1+d}/\sqrt{n}$ and $\kappa(A)^{1+3d/2}/\sqrt{M}$, respectively, where $\kappa(A)$ is the condition number of the finite element system matrix, $d$ is the spatial dimension, $n$ is the hidden-layer width, and $M$ is the number of training samples.
    Since $\kappa(A) = O(h^{-2})$ when $d=2$, these bounds grow rapidly under mesh refinement, suggesting potentially pessimistic scaling on fine meshes.
    In contrast, our analysis bounds the error in terms of the FEM error and the residual-based loss, which provides a more direct characterization of the approximation error.
\end{remark}

\section{Numerical Experiments}\label{sec:numerical_exp}

In this section, we present numerical experiments for the proposed residual-based convolutional method applied to elliptic PDEs.
We consider FE-CON with triangular elements (FE-T) and rectangular elements (FE-R).
A U-Net-style fully convolutional network is used, and the architectural details are provided in Appendix~\ref{app:architecture}.
In the decomposed formulation, each subproblem model uses the same architecture and capacity as the corresponding original model.
The models are trained without paired input--output data by minimizing residual-based loss functions. 

We use superscripts to indicate the discretization method, such as FE-T or FE-R.
Thus, $\hat{u}^{\mathrm{FE\text{-}T}}$ and $\hat{u}^{\mathrm{FE\text{-}R}}$ denote the approximate solutions obtained by FE-CON models with triangular and rectangular elements, respectively.
A symbol without a subscript, such as $\hat{u}$, denotes the approximate solution of the original model.
The subscript $s$ denotes the decomposed solution $\hat{u}_s = \hat{u}_1 + \hat{u}_2$, and $\hat{u}_1$ and $\hat{u}_2$ denote the approximate solutions of the two subproblems \eqref{Subproblem1} and \eqref{Subproblem2}, respectively.

Experiments are conducted on $N\times N$ grids with $N\in\{16,32,64,128\}$.
The accuracy of all models is evaluated using the relative error
\begin{align*}
\textrm{Rel}_{H^1}(v) &:= \frac{\|u_h^* - v\|_{H^1(\Omega)}}{\|u_h^*\|_{H^1(\Omega)}},
\end{align*}
where $u^*_h$ is the reference finite element solution computed on a sufficiently fine grid (approximately $1024\times1024$) in double-precision.

Unless otherwise stated, all experiments are conducted under the same experimental setting.
The Adam optimizer is used with an initial learning rate of $10^{-4}$, combined with a cosine decay schedule over the entire training horizon.
All experiments are trained for $10{,}000$ epochs with a batch size of $32$.
Training is performed in single-precision (float32), and model checkpoints are selected based on the minimum training loss.
All experiments are conducted on an Ubuntu 22.04 system equipped with an NVIDIA RTX A6000 GPU with 48 GB of memory, and the models are implemented in Python using JAX.

\subsection{Performance for the model problem}\label{ex:performance}

We first consider the Poisson problem \eqref{Poisson_problem} on the unit square domain, $\Omega=(0,1)^2$, with a Neumann condition on the left boundary and Dirichlet conditions on the remaining boundaries.
In this subsection, we do not aim to achieve optimal convergence.
Instead, we investigate the behavior of the models under a uniform training condition.
A detailed study of optimal convergence will be presented in the next subsection.

We use a fixed training set of $3{,}000$ samples and an independent test set of $1{,}000$ samples, and keep hyperparameter settings fixed across grid sizes and models.
The source term $f$ and boundary data ($g_D$, $g_N$) are sampled independently from randomized smooth sinusoidal functions of the form:
\begin{equation}\label{samples_Poisson}
    \phi(x,y)=a_1\sin(b_1 x+b_2 y)+a_2\cos(b_3 x+b_4 y),
\end{equation}
where the coefficients $a_1,a_2,b_1,\dots,b_4 \in [0,5]\subseteq\mathbb{R}$ are sampled independently.
This choice follows the setup used in \cite{MR4888707} and provides a convenient and widely used benchmark for operator learning methods.
This independent sampling avoids spurious correlations between interior and boundary inputs and exposes the models to fully varying mixed boundary conditions.

\begin{figure}[t!]
    \centering
    \begin{subfigure}{0.49\linewidth}
        \centering
        \includegraphics[height=6cm]{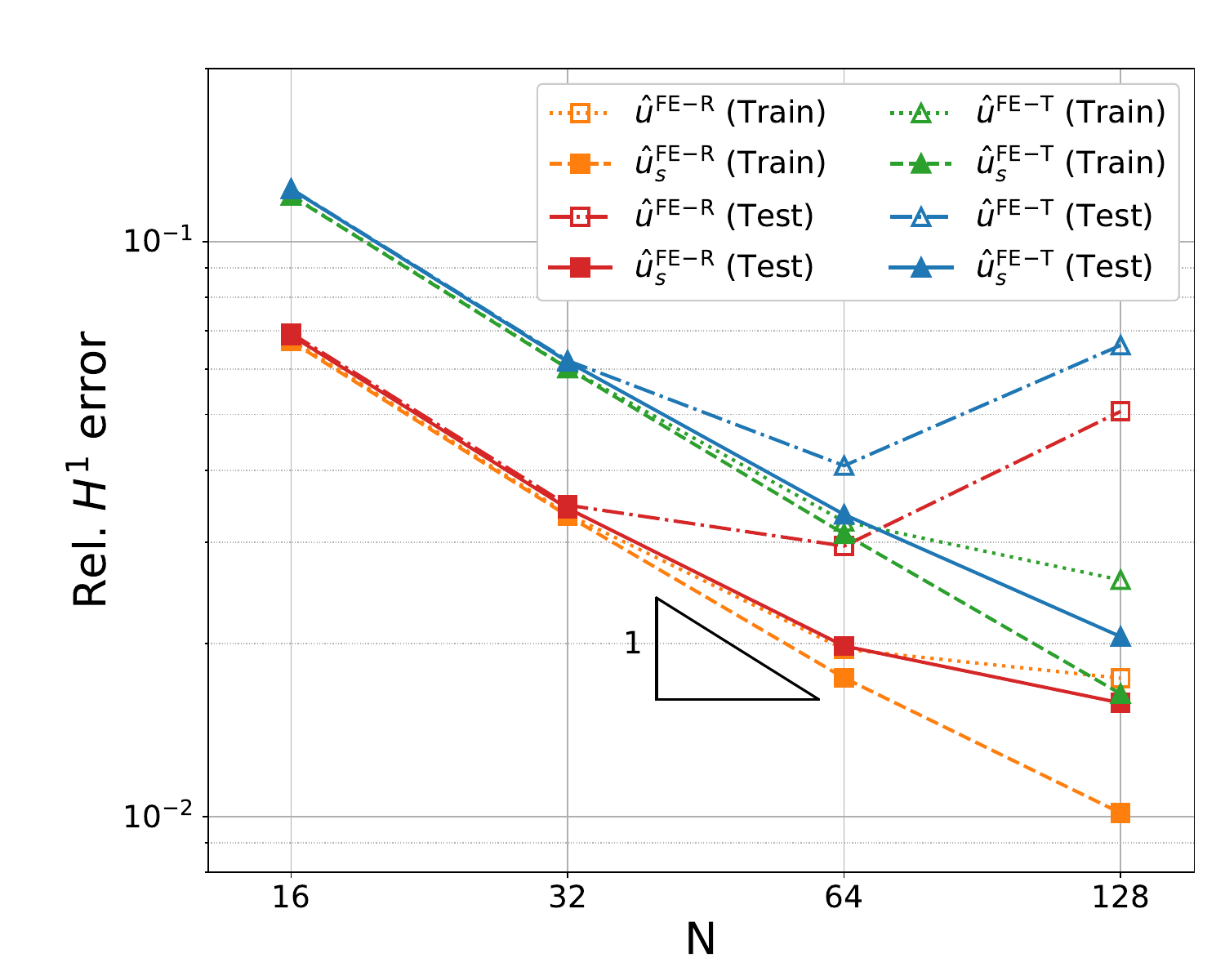}
        \caption{Relative error}
        \label{fig:ex1_error}
    \end{subfigure}
    \hfill
    \begin{subfigure}{0.49\linewidth}
        \centering
        \includegraphics[height=6cm]{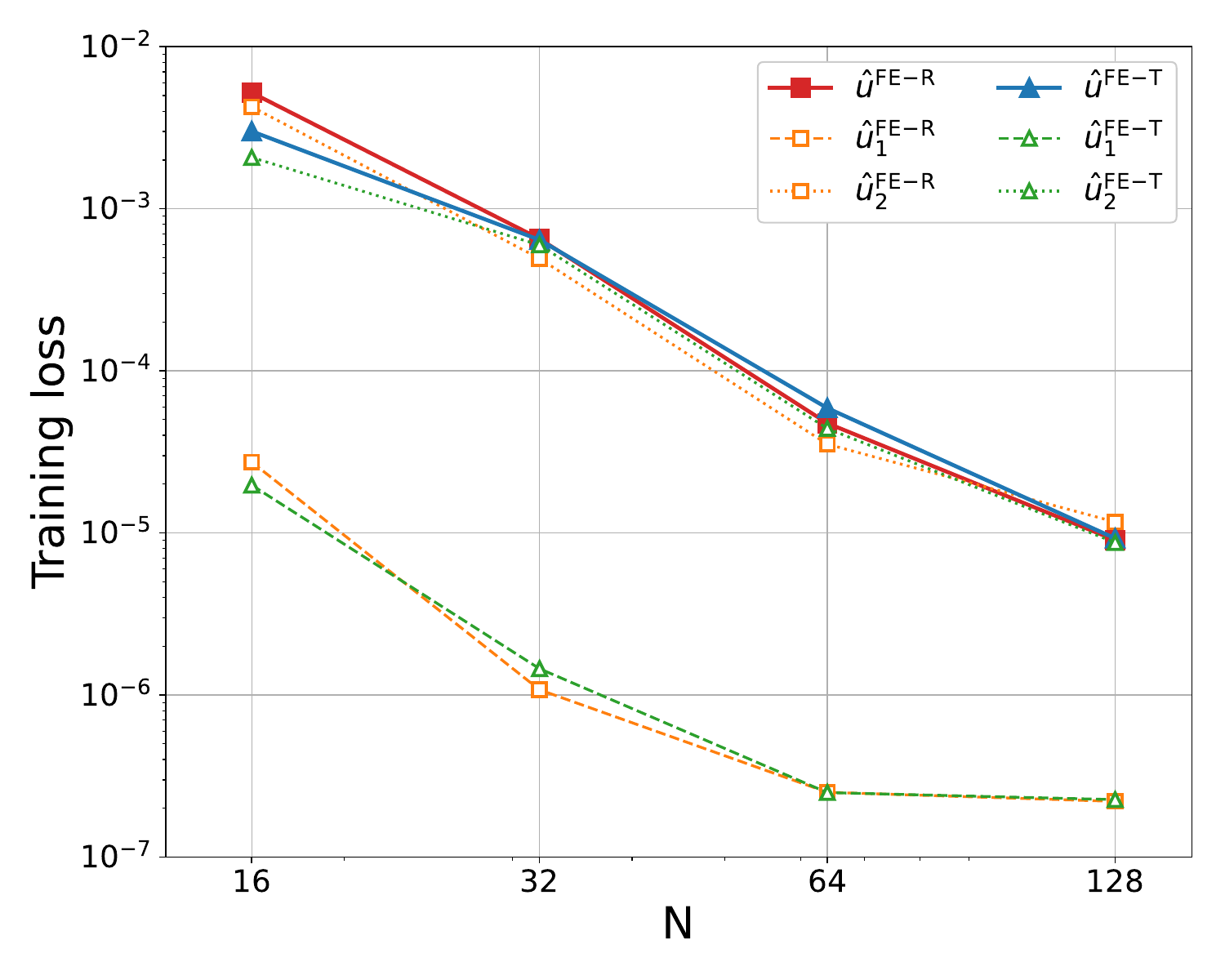}
        \caption{Training loss} 
        \label{fig:ex1_loss}
    \end{subfigure}
    \caption{Relative $H^1$ errors and training losses of FE-CON under grid refinement. 
    }
    \label{fig:ex1}
\end{figure}

Figure~\ref{fig:ex1_error} shows the convergence results.
On the training set, all models exhibit a clear reduction of the relative $H^1$ error as the grid is refined.
The decomposed formulation consistently achieves smaller errors than the original formulation.
On the test set, the errors are generally larger than those observed during training.
For the original formulation, while the test error decreases as the grid is refined up to $N=64$, it increases at $N=128$, which results in a noticeable gap between training and test errors.
In contrast, the decomposed formulation produces smaller test errors on fine grids, which indicates improved generalization under grid refinement.

Figure~\ref{fig:ex1_loss} shows the training loss for the original and decomposed formulations.
The loss of Subproblem 1 is significantly smaller than that of the original problem, whereas the loss corresponding to Subproblem 2 is comparable to that of the original problem.
This indicates that decomposition does not necessarily simplify the optimization process.
Nevertheless, the decomposed formulation consistently yields lower test errors across grid resolutions in Figure~\ref{fig:ex1}.
This suggests that its advantage arises not only from optimization, but also from learning structurally simpler sub-operators whose combination improves generalization of the full solution operator.

\subsection{Achieving optimal convergence and generalization} \label{numerical_verification}

In Section~\ref{ex:performance}, all models are trained with a fixed number of training samples and identical hyperparameters across grid sizes.
Under this setting, the training error does not always exhibit the optimal convergence rate predicted by classical FEM theory, and may even deteriorate as the grid is refined.
According to Corollary~\ref{cor_train} and Remark~\ref{rem:loss_scaling}, achieving the optimal convergence rate requires the training loss to decrease at a rate that depends explicitly on the grid size. 
More precisely, when the grid resolution $N$ is doubled
(approximately, $h \mapsto h/2$), training should continue until the loss satisfies a scaling condition of the form
\[
\mathcal{L}_{\mathrm{FEM}}(2N) \le \mathcal{L}_{\mathrm{FEM}}(N)/16.
\]

Figure~\ref{fig:FE_R_error_analysis_train} shows that, when this stopping criterion is enforced, the training error recovers the optimal convergence behavior under grid refinement.
The theoretically derived loss scaling ($\gamma = 4$) leads to consistent optimal convergence, while slightly weaker scalings still yield near-optimal rates.
In contrast, insufficient loss decay results in degraded convergence in agreement with the theoretical analysis.
Figure~\ref{fig:FE_R_error_analysis_test} shows the corresponding test errors for different loss scalings.
Unlike the training error, the test error is not minimized at $\gamma = 4$.
Instead, larger loss reduction leads to increased test error, which indicates overfitting under a limited number of training samples.
This behavior is consistent with the generalization behavior discussed below.

\begin{figure}[t!]
    \centering
    \begin{subfigure}{0.49\linewidth}
        \centering
        \includegraphics[height=6cm]{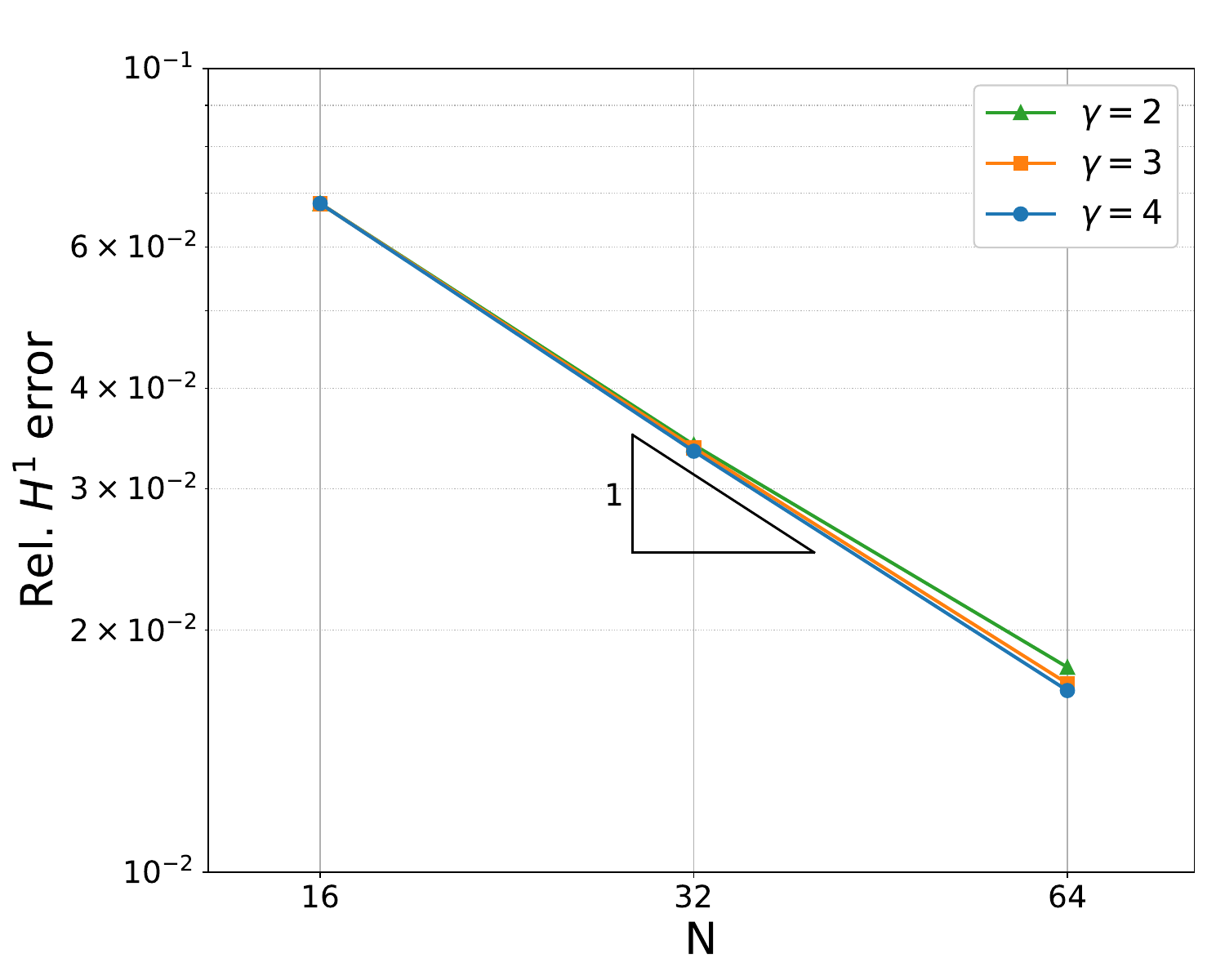}
        \caption{Training error} \label{fig:FE_R_error_analysis_train}
    \end{subfigure}
    \hfill
    \begin{subfigure}{0.49\linewidth}
        \centering
        \includegraphics[height=6cm]{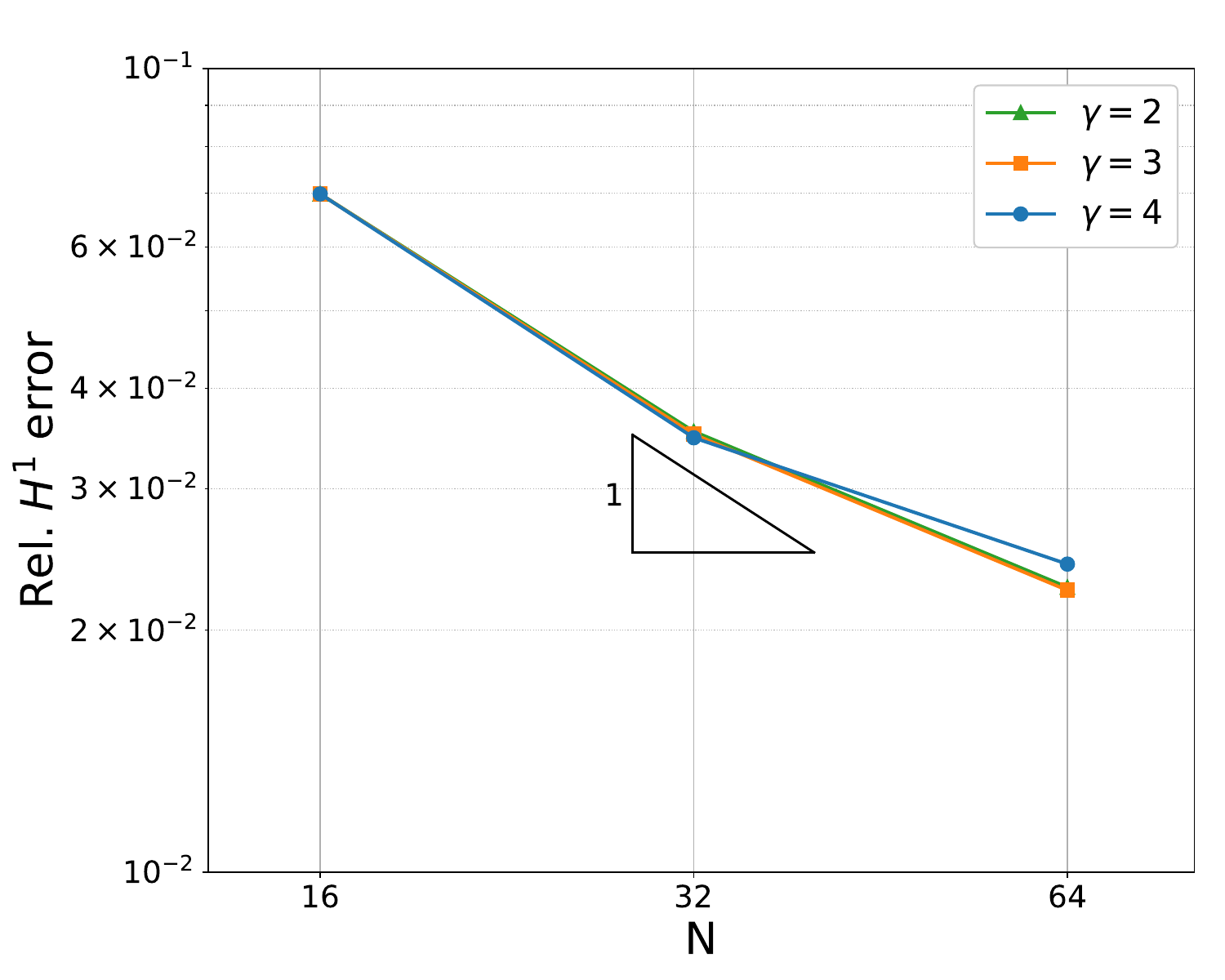}
        \caption{Test error} \label{fig:FE_R_error_analysis_test}
    \end{subfigure}
    \caption{Relative $H^1$ errors of FE-CON under different loss scalings $(O(h^\gamma))$.}
    \label{fig:error_analysis}
\end{figure}

\begin{figure}[t!]
    \centering
    \begin{subfigure}{0.49\linewidth}
        \centering
        \includegraphics[height=6cm]{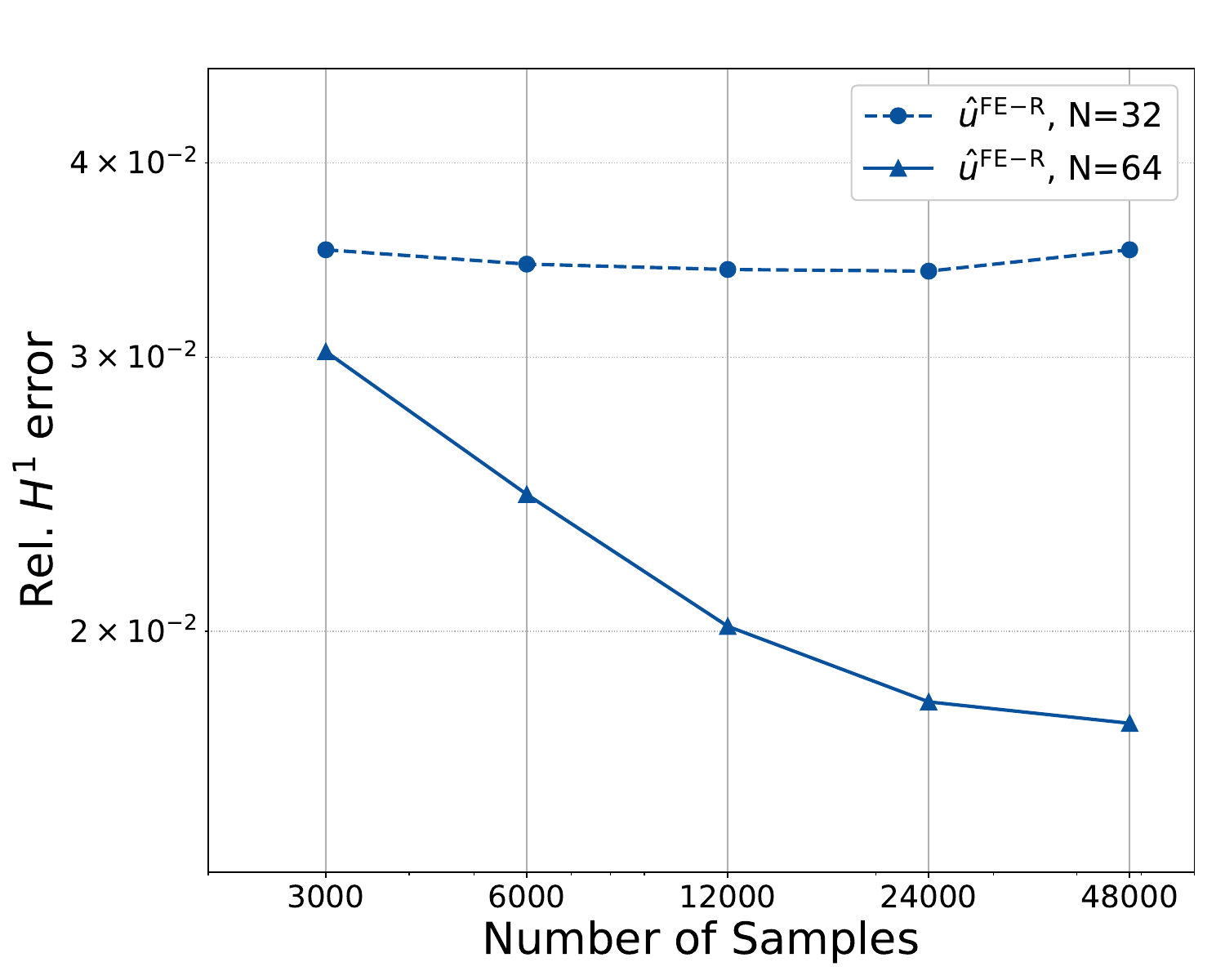}
        \caption{Test error}
        \label{fig:generalization_error_test}
    \end{subfigure}
    \hfill
    \begin{subfigure}{0.49\linewidth}
        \centering
        \includegraphics[height=6cm]{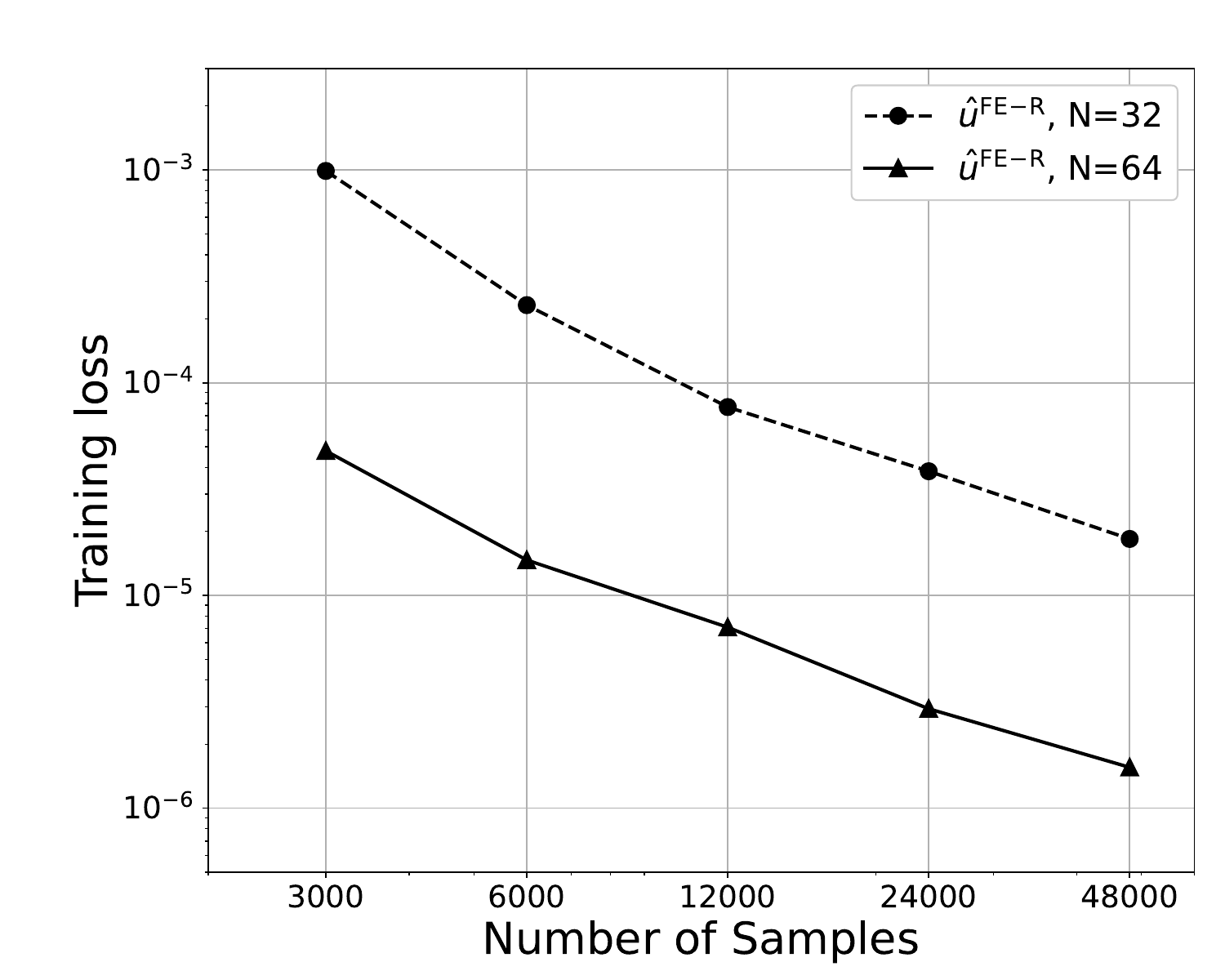}
        \caption{Training loss} \label{fig:generalization_error_train}
    \end{subfigure}
    \caption{Relative $H^1$ errors on the test set and training loss versus the number of training samples.}
    \label{fig:generalization}
\end{figure}

Building on the theoretical result for finite element operator learning methods in \cite{hong2024error}, which states that the generalization error scales as $O(N_s^{-1/2})$,
our results in Figure~\ref{fig:generalization} show a consistent dependence on the sample size.
In our training setting, the training set size is fixed at 3,000 samples.
As a result, the training loss can be reduced to the level required by the stopping criterion, and the training error can recover the optimal convergence behavior under grid refinement.
However, this does not necessarily guarantee the optimal convergence rate on the test set, as also observed in Figures~\ref{fig:ex1} and \ref{fig:error_analysis}.
One key reason is that increasing the grid parameter $N$ enlarges the U-Net architecture and increases the number of trainable parameters.
Therefore, a fixed and relatively small training set may be insufficient to constrain the model appropriately at finer resolutions.

Figure~\ref{fig:generalization} supports this interpretation. 
The training loss decreases consistently as the number of training samples increases for both $N = 32$ and $N = 64$.
However, the test error shows different behaviors depending on the resolution.
For $N=32$, the test error remains nearly unchanged, whereas for $N=64$, it decreases consistently as the number of training samples increases. 
In particular, for sufficiently large sample sizes, the test error at $N = 64$ approaches half of the error level observed at $N=32$, which is consistent with the expected optimal convergence rate.
These results suggest that achieving the optimal convergence rate on the test set requires not only sufficient optimization, as ensured by the stopping criterion, but also a training sample size that grows with the effective problem complexity at larger $N$.
This complexity can be influenced by factors such as the increase in the number of trainable parameters and the conditioning of the underlying FEM system.

\begin{remark}[Comparison with \cite{hong2024error}]
    The different theoretical results (see Remark~\ref{rmk:hong2024}) are also reflected in numerical behavior.
    In \cite{hong2024error}, numerical experiments for a one-dimensional convection-diffusion equation show that the theoretical upper bounds begin to increase even for relatively coarse meshes.
    In contrast, our experiments are conducted on two-dimensional problems with the number of elements scaling as  $(N-1)^2$, yet we observe stable and monotone convergence even on fine grids.
    One plausible factor behind the different empirical behavior is optimization difficulty.
    In \cite{hong2024error}, the number of trainable parameters typically grows with both network width and depth, which can make it harder to drive the training loss to the level needed for stable behavior under mesh refinement.
    In contrast, we use a U-Net architecture without fully connected layers and further reduce the learning complexity via subproblem decomposition, which together improve training stability on fine meshes.
\end{remark}

\subsection{Training efficiency via problem decomposition}

We investigate the training efficiency of the decomposed formulation using FE-CON with rectangular elements (FE-R) as a representative example; similar trends are observed for the model with triangular elements (FE-T).
In this subsection, the training and test sets are constructed differently to evaluate the training efficiency of the decomposed formulation.

We construct two independent training sets corresponding to the two subproblems, one for the pair $(f, \ g_N)$ and the other for $g_D$, each of size 100.
These define a pool of 10,000 possible joint input combinations.
From this pool, 5,000 combinations are sampled to train the original formulation, while the remaining 5,000 disjoint combinations are used as a common test set.
Each subproblem model is trained independently using the corresponding training set of 100 samples and does not observe any joint input combinations during training.
This reflects the structurally simpler learning tasks associated with each subproblem, since each model depends only on a separated component of the input data.

\begin{figure}[t!]
    \centering
    \includegraphics[height=6cm]{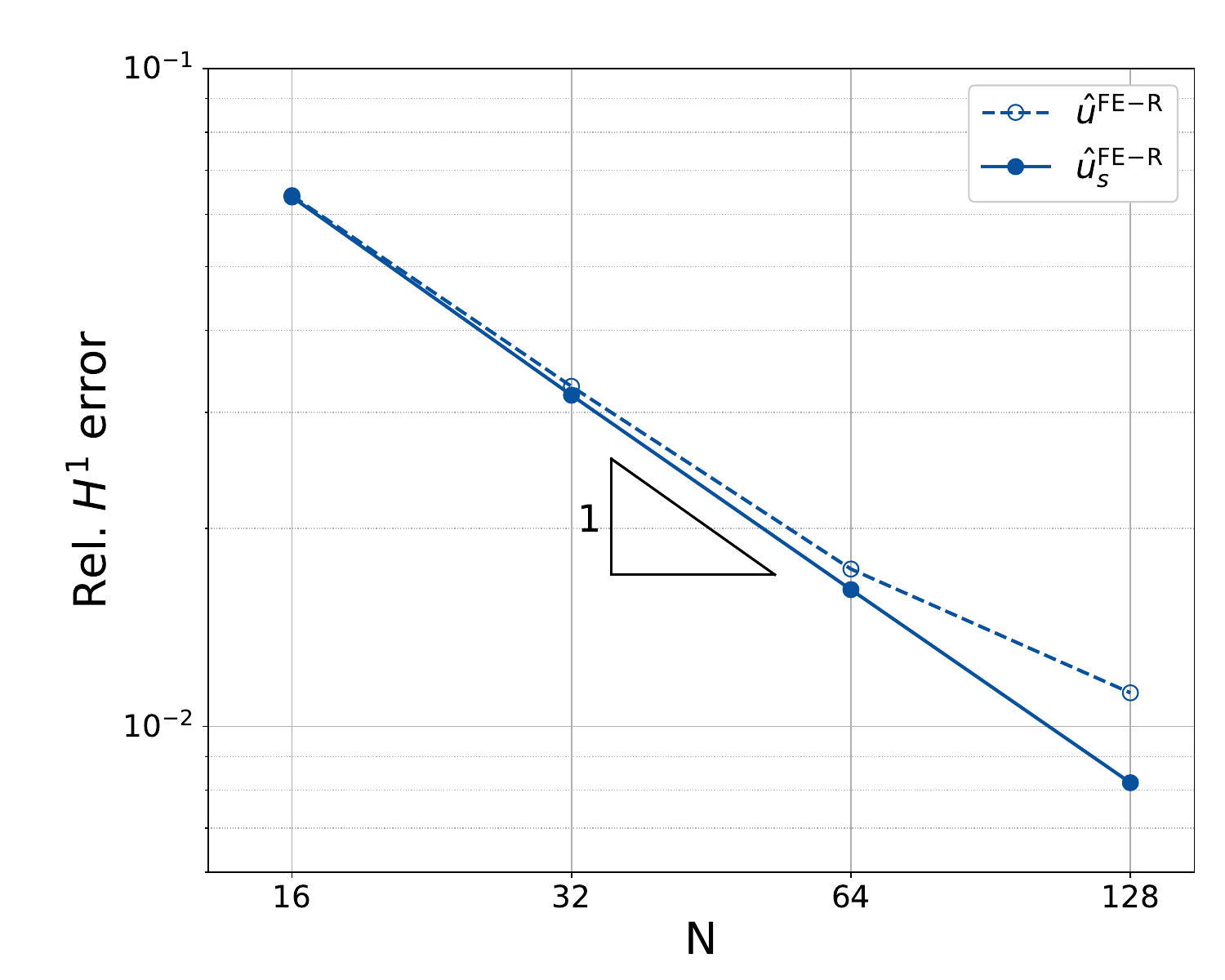}
    \caption{Decomposition-based learning improves data efficiency for FE-CON with rectangular elements (FE-R) by composing independently learned subproblem predictions at inference.}
    \label{fig:decomposition}
\end{figure}

Figure~\ref{fig:decomposition} shows that, despite using significantly fewer training samples, the decomposed formulation consistently achieves lower relative $H^1$ errors across all grid resolutions.
This improvement can be explained by the combinatorial structure induced by decomposition.
Training on independent input sets implicitly covers the joint input combinations generated by their composition.
As a result, the test inputs can be regarded as being effectively included in the training samples of the decomposed formulation.

Importantly, this effect does not apply to the experiments in Section~\ref{ex:performance}, where the test inputs are not implicitly covered by the training samples of the decomposed formulation.
The improved performance in that setting indicates that the benefit of decomposition is not limited to combinatorial coverage. 

These results show that problem decomposition can substantially improve training efficiency by exploiting the compositional structure of the solution operator.
Although the decomposed formulation requires multiple sub-operator models, each subproblem separates one part of the input data and therefore has a simpler input-output structure.
This leads to more stable and reliable training than directly learning the full operator from a large set of training samples.
The decomposed approach learns sub-operators independently and composes them during inference.
As a result, it uses a fixed number of training samples more effectively and provides a practical alternative when training with a large dataset is costly or unstable.

\subsection{Application to complex geometry}\label{subsec:complex_geometry}

We further examine the proposed method on a domain with a more complex geometry.
This experiment differs from that in Section~\ref{ex:performance} in two aspects.
First, we consider a square domain with a square hole,
\begin{equation}\label{eq:complex_domain}
    \Omega = (0,1)^2 \setminus [0.4,0.6]^2.
\end{equation}
Second, nonhomogeneous Dirichlet boundary conditions are imposed on both the outer and inner boundaries.

\begin{figure}[t!]
    \centering
    \includegraphics[height=9cm]{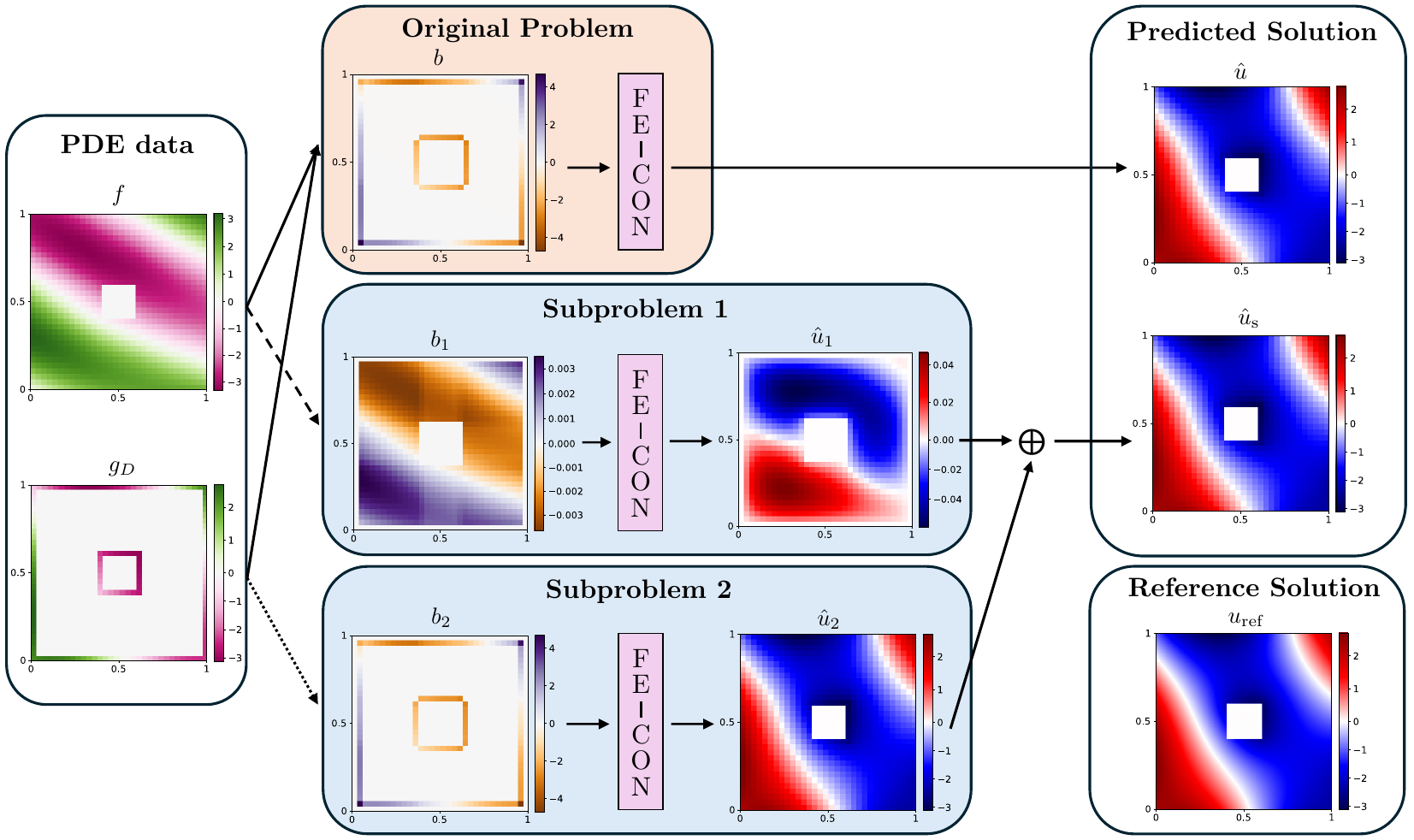}
    \caption{Computational workflow of FE-CON on a complex geometry~\eqref{eq:complex_domain} for $N=32$.}
    \label{fig:workflow_FE}
\end{figure}

Figure~\ref{fig:workflow_FE} illustrates the computational workflow of FE-CON.
In FE-CON, a single-channel image corresponding to the right-hand side vector in the FE system is used as the input.
As shown in the figure, when the magnitudes of $f$ and $g_D$ are comparable, the original problem becomes more difficult to learn.
This difficulty arises from the construction of the input.
According to \eqref{b1}, the contributions of $f$ and $g_N$ are obtained through integrations over the domain $\Omega$ and the Neumann boundary $\partial\Omega_N$, respectively.
Since the basis functions have local support, the resulting integrals decrease as $N$ increases.
In contrast, the integration in \eqref{b2} is independent of the mesh size $h$, and hence it does not decay as $N$ increases.
As a result, $\boldsymbol{b}_2$ becomes dominant for finer grids, and $\boldsymbol{b}$ and $\boldsymbol{b}_2$ appear nearly indistinguishable in Figure~\ref{fig:workflow_FE}.
By decomposing the original problem into subproblems and treating these components separately, the proposed decomposition avoids this unfavorable scaling imbalance.
Consequently, the decomposition yields a more balanced learning problem and leads to improved numerical robustness.

\begin{figure}[t!]
  \centering
  \begin{subfigure}{0.49\linewidth}
    \centering
    \includegraphics[height=6cm]{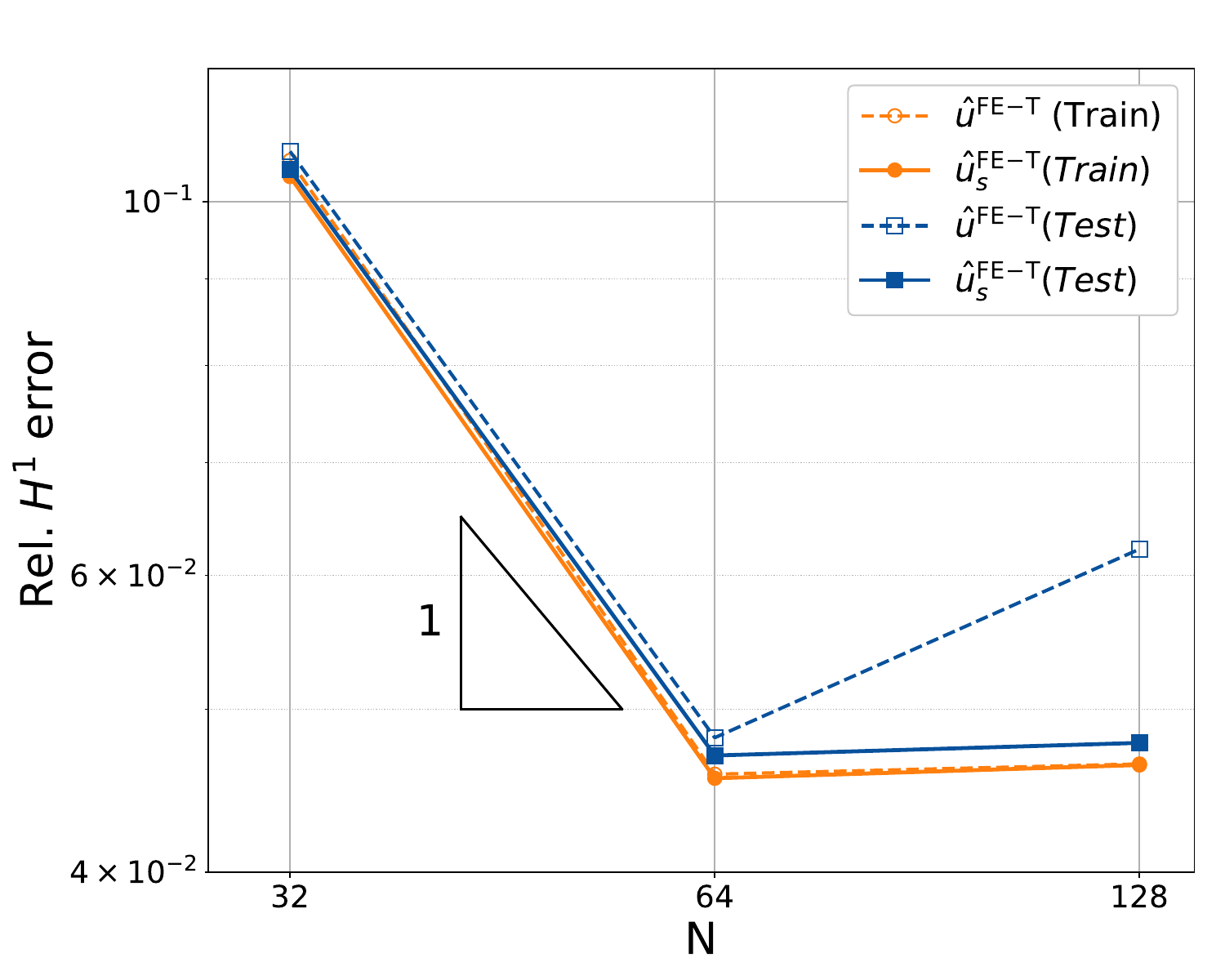}
    \caption{Relative $H^1$ error}
    \label{fig:complex_geometry_test_loss}
  \end{subfigure} 
  \hfill
  \begin{subfigure}{0.49\linewidth}
    \centering
    \includegraphics[height=6cm]{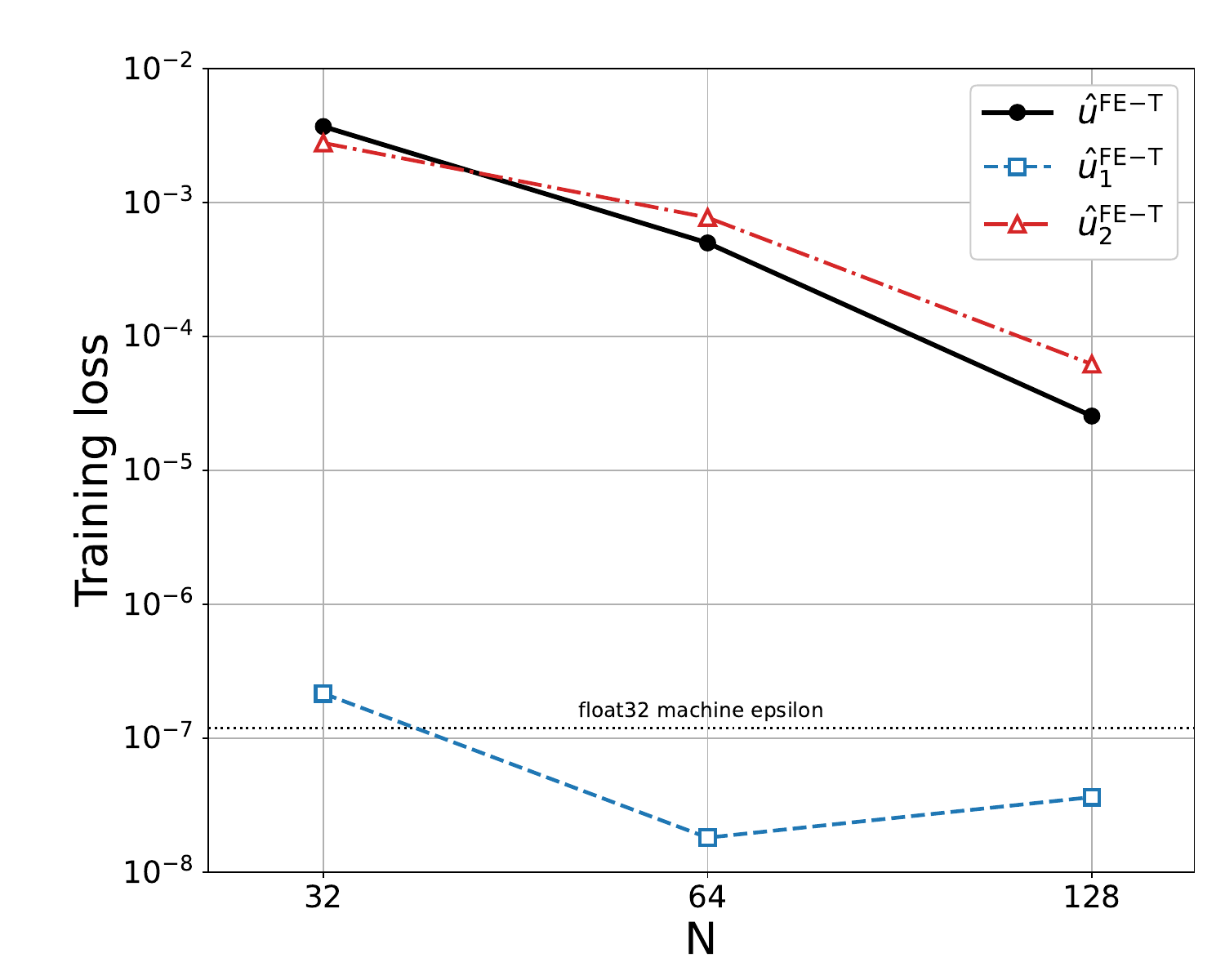}
    \caption{Training loss}
    \label{fig:complex_geometry_train_loss}
  \end{subfigure}
  \caption{Relative $H^1$ error and residual-based training loss across grid sizes on the square-hole domain.}
  \label{fig:complex_geometry}
\end{figure}

Figure~\ref{fig:complex_geometry_test_loss} shows the relative $H^1$ error with respect to the reference finite element solutions.
As $N$ increases from 32 to 64, both the training and test errors decrease for all models.
When $N$ increases further to 128, the decomposed formulation shows only a slight increase in the error, whereas the original formulation exhibits a noticeable increase in the test error.
Figure~\ref{fig:complex_geometry_train_loss} shows the training loss.
For the decomposed formulation, the loss for subproblem 1 reaches the float32 precision limit for $N\ge 64$, which indicates that further reduction is constrained by finite-precision arithmetic.
The loss for subproblem 2 continues to decrease, but this does not lead to further improvement in the $H^1$ error at finer resolutions.

The convergence behavior is also influenced by the regularity of the solution.
Even when the source term and Dirichlet boundary conditions are smooth, the presence of reentrant corners introduces singularities in the solution.
As a result, the solution may not have sufficient regularity for optimal convergence. The observed convergence rate from $N=32$ to $N=64$ exceeds the asymptotic rate, which indicates that the experiment is in a pre-asymptotic regime.
On finer grids, the error becomes dominated by the singularity, and further refinement yields only marginal improvement or slight degradation.

Finally, the increase in the test error of the original formulation at $N = 128$ is primarily attributed to generalization rather than optimization.
The number of training samples is fixed across $N$, while the model complexity increases with $N$.
This can enlarge the generalization gap, as shown in Figure~\ref{fig:ex1}.
This issue can be reduced by increasing the number of training samples at larger $N$, as shown in Figure~\ref{fig:generalization}.

\subsection{Application to Helmholtz equation}\label{subsec:helmholtz}

To assess the applicability of the proposed method beyond the Poisson problem, we consider the Helmholtz equation of the form \cite{mcclenny2020self, wang2021understanding},
\begin{subequations}\label{Helmholtz_problem}
\begin{align}
\Delta u + \kappa^2 u &= f \ \ \quad \text{in } \ \ \Omega \label{Helmholtz_f}, \\
u &= g_D \quad  \text{on } \ \partial\Omega\label{Helmholtz_D},
\end{align}
\end{subequations}
where $\Omega=(0,1)^2$ and $\kappa=1$.
The forcing term is chosen so that the exact solution is
\begin{equation}\label{eq:helmholtz_u}
u(x,y) = \sin(a_1 \pi x)\sin(a_2 \pi y),
\end{equation}
with  $a_1,\ a_2 \in [1,20]$ sampled independently. The Dirichlet boundary condition is prescribed by the trace of the exact solution, $g_D = u|_{\partial\Omega}$. This setting produces solutions with highly oscillatory behavior over a wide range of frequencies.
For evaluation, the exact solution is used in place of $u_h^*$ for computing the relative $H^1$ error.

\begin{figure}[t!]
    \centering
    \includegraphics[height=6cm]{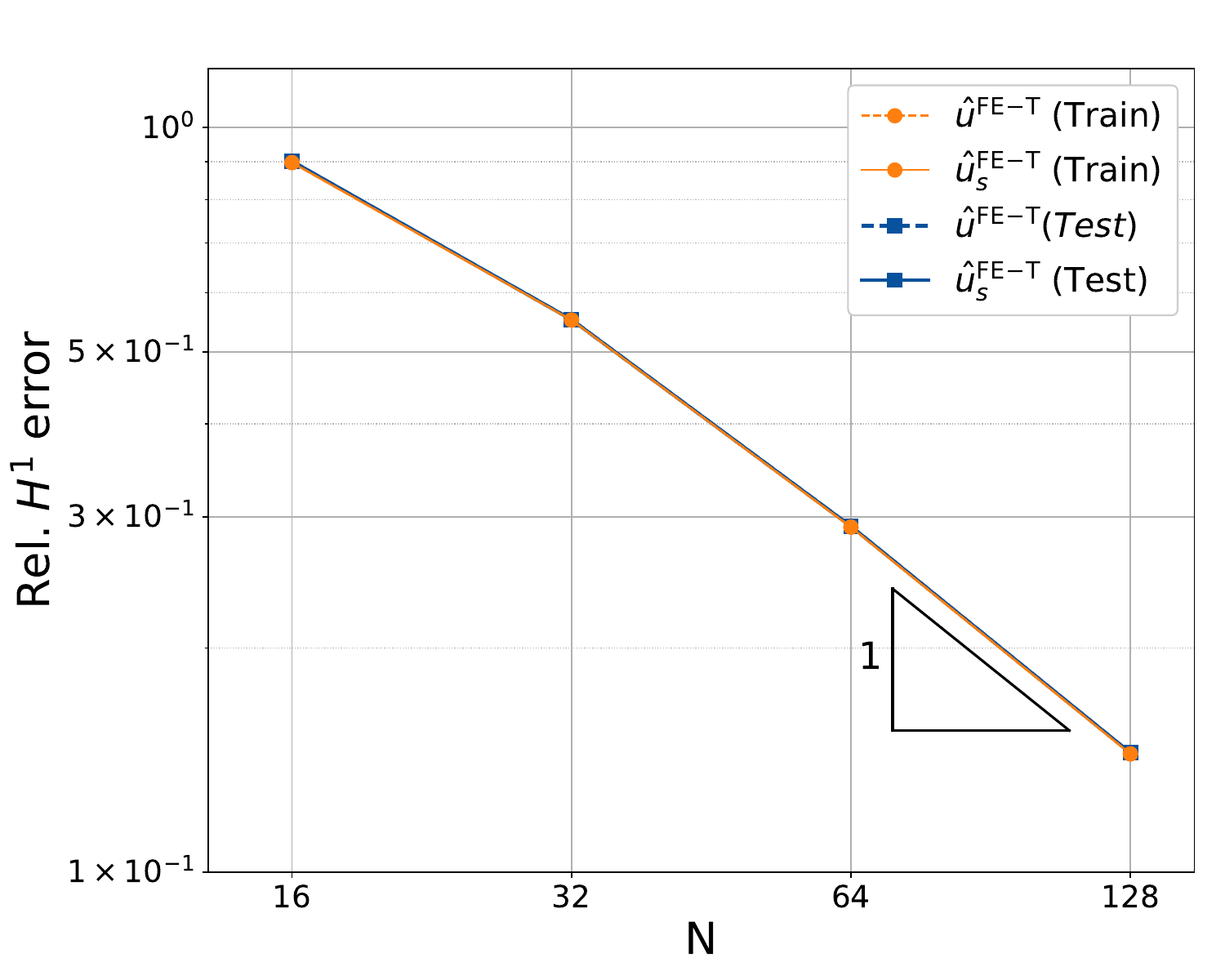}
    \caption{Relative $H^1$ error for the Helmholtz problem \eqref{Helmholtz_problem} under uniform grid refinement.}
    \label{fig:helmholtz_h1}
\end{figure}

Figure~\ref{fig:helmholtz_h1} shows the relative $H^1$ error for the Helmholtz problem under uniform grid refinement.
The errors for both the training and test sets decrease consistently as the grid is refined, and their magnitudes remain comparable across all resolutions.
These results show that the proposed method captures the solution operator without a noticeable generalization gap.

\begin{figure}[t!]
    \centering
    \includegraphics[height=9cm]{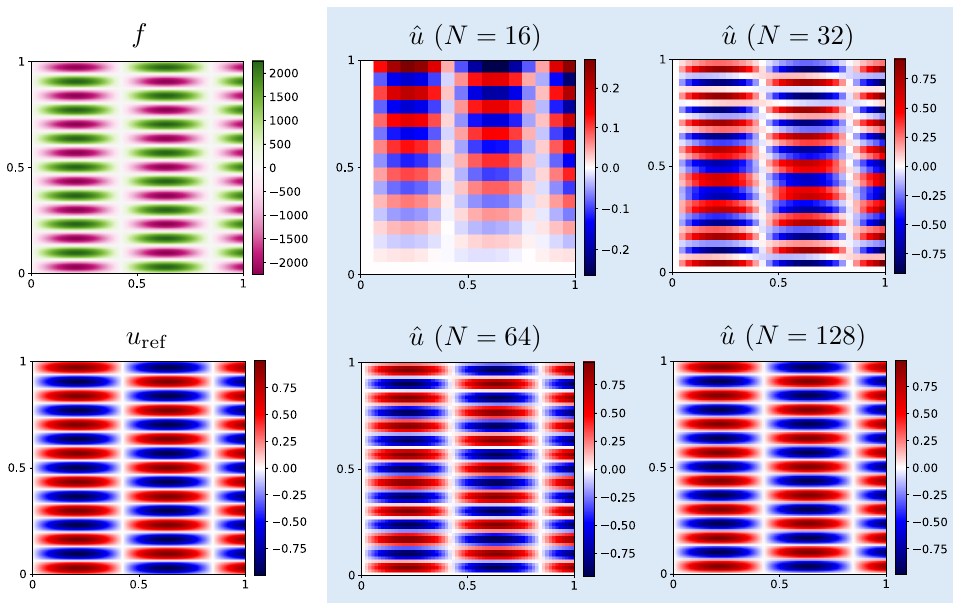}
    \caption{FE-CON approximations for the Helmholtz problem at $N = 16, 32, 64, 128$.}
    \label{fig:profile_Helmholtz}
\end{figure}

Figure~\ref{fig:profile_Helmholtz} shows representative approximate solutions produced by FE-CON for $N = 16, 32, 64, 128$.
The solution exhibits anisotropic oscillatory behavior, with slow variation in the $x$-axis direction and rapid oscillations in the $y$-axis direction.
Such strongly oscillatory features are difficult to resolve on coarse grids.
For $N = 16$ and $N = 32$, the predicted solutions do not accurately capture these oscillations.
As the grid is refined, the quality of the predicted solutions improves significantly.
For $N = 64$ and $N = 128$, the approximations closely match the reference solution and correctly represent the oscillatory structure.
These results indicate that accurate approximation of highly oscillatory solutions requires sufficient spatial resolution.
The present experiment demonstrates that the proposed method maintains its effectiveness under grid refinement and is capable of resolving fine-scale features in Helmholtz-type problems.

\subsection{Extension to nonlinear problem}\label{subsec:ginzburg}
To further examine the applicability of the proposed method beyond linear elliptic equations, we consider the Ginzburg--Landau equation
\begin{subequations}\label{ginzburg_equation}
\begin{align}
-\epsilon\Delta u - u + u^3 &= f \ \quad \text{in } \ \ \Omega \label{ginzburg_f}, \\
u &= 0 \ \quad  \text{on } \ \partial\Omega\label{ginzburg_D},
\end{align}
\end{subequations} where $\Omega=(0,1)^2$. Unlike the linear problems considered in the previous sections, this equation contains a cubic nonlinear term.
Therefore, the decomposition strategy based on linear superposition is not applicable in this setting.
Instead, we train a single FE-CON model by directly minimizing the nonlinear finite element residual.

\begin{figure}[t!]
  \centering
  \begin{subfigure}{0.49\linewidth}
    \centering
    \includegraphics[height=6cm]{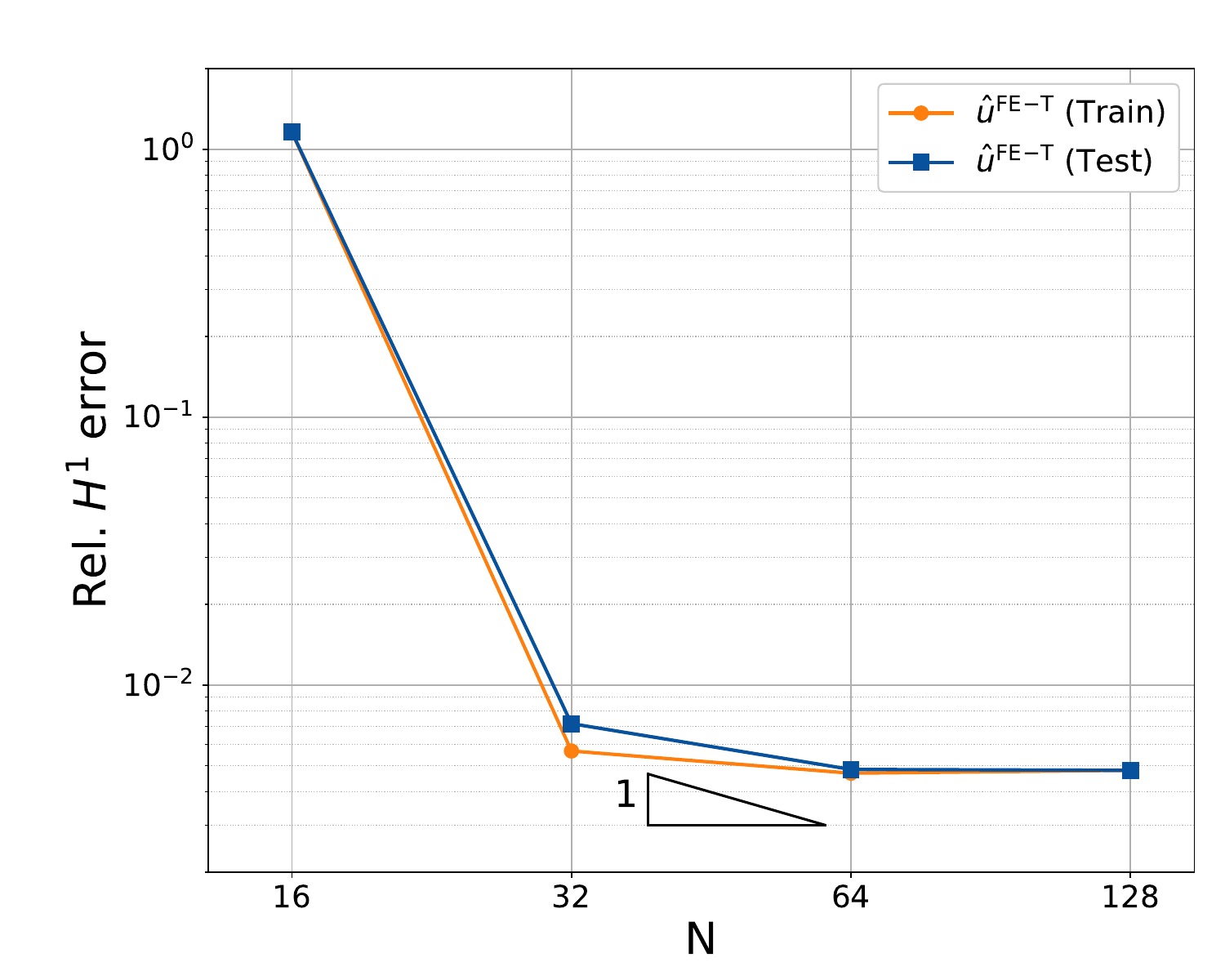}
    \caption{Relative $H^1$ error}
    \label{fig:ginzburg_landau_h1}
  \end{subfigure} 
  \hfill
  \begin{subfigure}{0.49\linewidth}
    \centering
    \includegraphics[height=6cm]{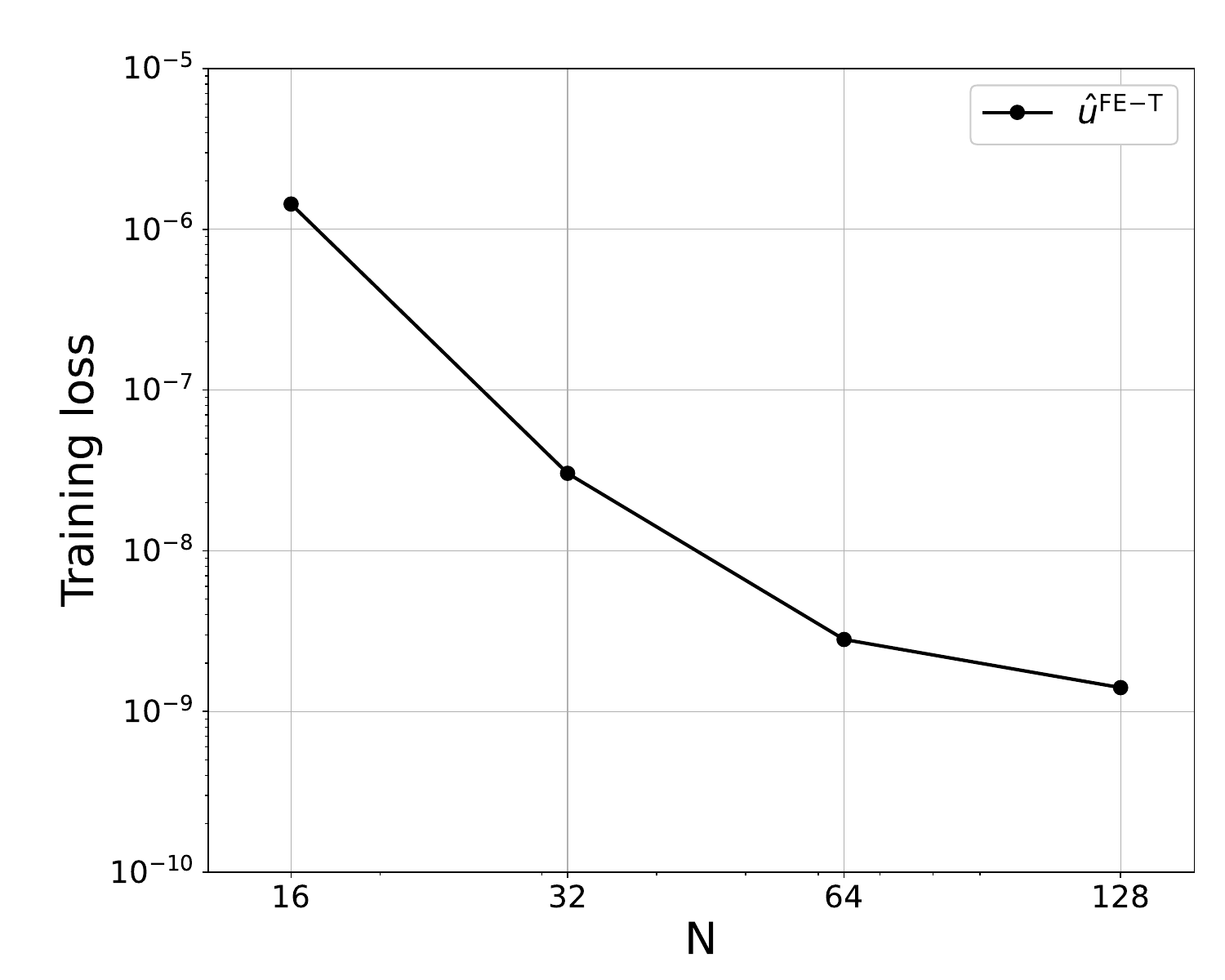}
    \caption{Training loss}
    \label{fig:ginzburg_landau_train_loss}
  \end{subfigure}
  \caption{Relative $H^1$ error and training loss for the Ginzburg--Landau problem under uniform grid refinement.}
  \label{fig:ginzburg_landau}
\end{figure}

In this experiment, we set $\epsilon = 0.01$ and use the same hyperparameters for training as in the previous sections, without enforcing the training strategy for optimal convergence.
Figure~\ref{fig:ginzburg_landau_h1} shows the relative $H^1$ errors under uniform grid refinement. 
The error decreases from $N=16$ to $N=32$ and remains stable for fine grids.
These results demonstrate that the proposed residual-based framework extends naturally to nonlinear problems.

\begin{table}[t]
\centering
\caption{Computational time comparison between the proposed method and the classical FEM solver.
For the Poisson problem, Iter. denotes the number of CG iterations, while for the Ginzburg-Landau problem, it denotes the number of Newton steps (Newt.) and the total CG iterations accumulated over all Newton steps.}
\label{tab:online_runtime_compare}
\renewcommand{\arraystretch}{1.15}

\begin{tabular}{c c ccc ccc}
\toprule
&
&
\multicolumn{3}{c}{Poisson}
&
\multicolumn{3}{c}{Ginzburg--Landau} \\
\cmidrule(lr){3-5}
\cmidrule(lr){6-8}

$N$
& Proposed (ms)
& FEM (ms)
& Speedup
& Iter.
& FEM (ms)
& Speedup
& Iter. (Newt./CG) \\
\midrule

16
& 0.525
& 1.785
& 3.40$\times$
& 51
& 292.3
& 557.1$\times$
& 15 / 314 \\

32
& 0.583
& 3.292
& 5.65$\times$
& 104
& 372.2
& 638.4$\times$
& 21 / 704 \\

64
& 1.165
& 6.672
& 5.73$\times$
& 208
& 417.0
& 357.9$\times$
& 18 / 1370  \\

128
& 5.705
& 21.50
& 3.77$\times$
& 404 
& 2116.0
& 370.9$\times$
& 18 / 2786 \\

\bottomrule
\end{tabular}
\end{table}

We further compare the computational cost of the proposed method with a classical FEM solver.
For linear problems, the FEM solver requires solving a linear system using either a direct method or an iterative method.
For nonlinear problems, the FEM approach requires repeated linearizations, which lead to a sequence of modified linear systems that are solved iteratively.
In contrast, once trained, the proposed method produces the solution through a single forward pass without any iterative procedure.
This leads to a reduction in the computational time.
In this experiment, we use the conjugate gradient (CG) method for linear solves and Newton's method for nonlinear iterations, with tolerances fixed at $10^{-6}$ for both.
The reported computational times are averaged over 1,000 independent runs for each case.
As shown in Table~\ref{tab:online_runtime_compare}, the proposed method reduces computational time compared with the classical FEM solver, especially for the nonlinear problem.
Here, the reported number of CG iterations denotes the total number accumulated over all Newton steps for the Ginzburg-Landau problem, while for the Poisson problem, it denotes the number of CG iterations for a single linear solve.

This property makes the proposed method attractive as a surrogate solver in multilevel frameworks.
In particular, it can be used to approximate coarse-grid solutions in two-grid methods, where fast and repeated evaluations are required.

\section{Conclusion}\label{sec:conclusion}
In this work, we proposed a residual-based convolutional approach for solving partial differential equations.
The method constructs a surrogate solver that directly maps input data to the corresponding solution by minimizing discretized residuals, without requiring paired input-output data.
We established a direct connection between the decay of the training loss and the convergence of the predicted solution.
The analysis provides explicit training criteria to achieve the optimal convergence rate under grid refinement.
These results offer practical guidelines for training residual-based operator models.
From a computational perspective, the proposed method avoids iterative solves and produces solutions through a single forward pass once training is completed.
This reduces computational cost, particularly for nonlinear problems, where classical methods require repeated linearizations and the corresponding solution of linear systems.
This suggests that the proposed method can be used as a surrogate solver in multilevel frameworks, such as two-grid methods.
We also introduced a decomposition strategy that separates the contributions of different input components.
This decomposition improves training efficiency and enhances generalization by enabling the model to learn simpler sub-operators.
Numerical experiments demonstrated that the proposed method is stable and accurate under grid refinement.
It is applicable to complex geometries, and it is capable of resolving fine-scale features in highly oscillatory problems such as the Helmholtz equation.
It is also shown that the decomposed formulation consistently outperforms the original formulation, even when trained with significantly fewer samples.
In addition, the method extends naturally to nonlinear equations.
Future work includes extensions to more complex geometries, higher-order approximations, and broader classes of nonlinear and multiphysics problems.

\bibliographystyle{amsplaindoi}
\bibliography{references}

\newpage
\appendix
\renewcommand{\thetable}{A.\arabic{table}}

\section{Network Architecture}\label{app:architecture}
We use a U-Net-style fully convolutional neural network~\cite{ronneberger2015u}.
The network maps input images to output fields defined on the same spatial grid.

\begin{figure}
    \centering
    \includegraphics[height=6.5cm]{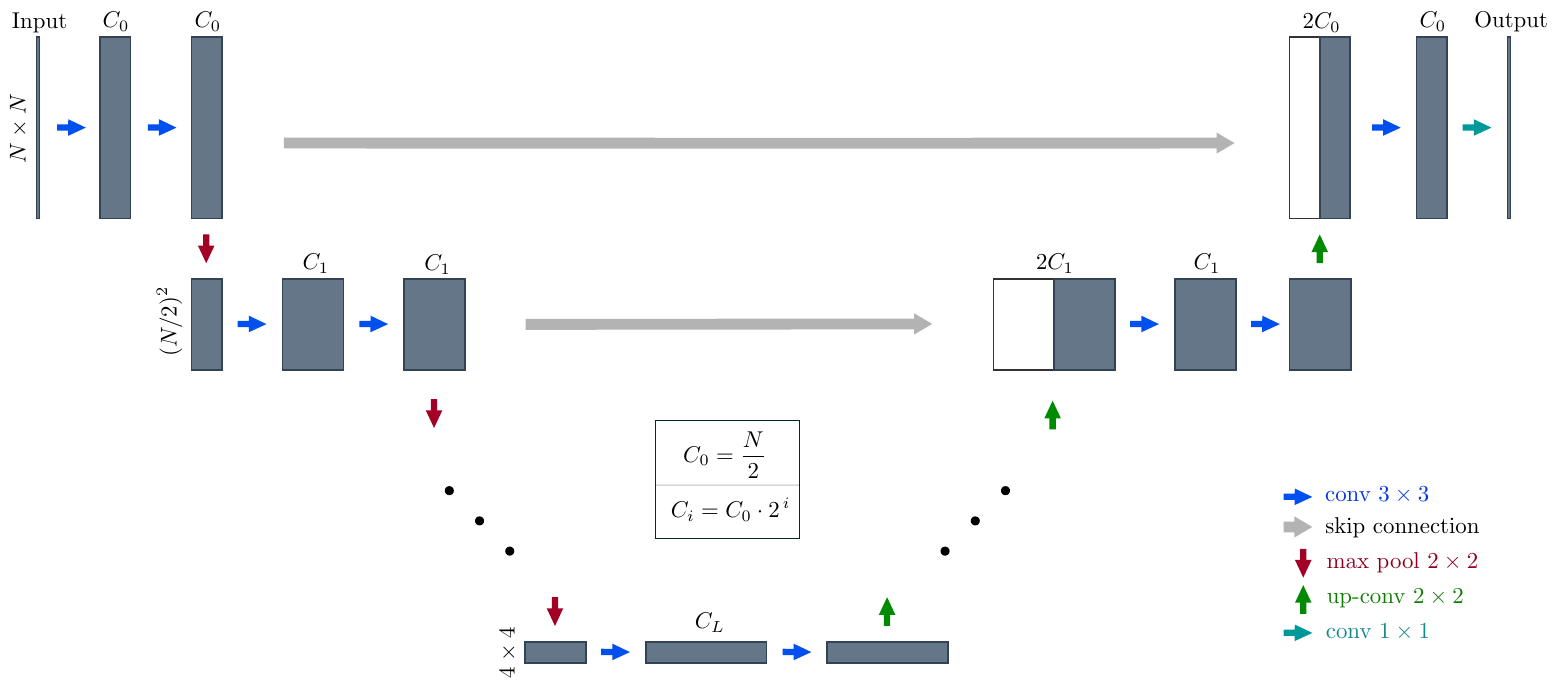}
    \caption{U-Net-style fully convolutional network architecture used in the experiments.}
    \label{fig:architecture}
\end{figure}

Figure~\ref{fig:architecture} shows the network architecture.
The architecture consists of an encoder, a bottleneck, and a decoder.
The encoder contains $L$ blocks, where each block applies two $3\times3$ convolutional layers with ReLU activations followed by a $2\times2$ max-pooling operation.
The number of channels at level $i=0,\dots,L-1$ is given by
$$C_i = C_0 \cdot 2^{\,i}.$$
At the bottleneck, two convolutional layers with channel dimension $C_0 \cdot 2^{L}$ are applied.
The decoder reflects the encoder. 
At each level, feature maps are upsampled using $2\times2$ transposed convolution, concatenated with the corresponding encoder features via skip connections, and processed by two $3\times3$ convolution layers with ReLU activations.
A final $1\times1$ convolution produces a single-channel output.
The model capacity is scaled with the grid resolution $N$ according to
$$C_0 = \frac{N}{2}, \qquad
N = 2^{L+2}, \qquad L \in \{2, 3, 4, 5\},$$
so that the spatial resolution at the bottleneck layer is $4 \times 4$ for all grid resolutions considered.

\section{Finite Difference Discretizations}
In this section, we briefly review a standard finite difference discretization on a uniform grid.
For simplicity, we consider a uniform grid with $x_i < x_{i+1}$, $y_j < y_{j+1}$, and $h := x_{i+1} - x_i = y_{j+1} - y_j$.
Let $u_{i,j} = u(x_i, y_j)$. 
Then, finite difference approximations to the Laplacian $\Delta$ are given by:
\begin{align}
    \text{5-point stencil:} & \quad \Delta_h u_{i,j} = \frac{u_{i-1,j} + u_{i+1,j} + u_{i,j-1} + u_{i,j+1} - 4u_{i,j}}{h^2} \label{fd_5}, \\[2ex]
    \text{9-point stencil:} & \quad \Delta_h u_{i,j} = \frac{u_{i-1,j-1} + u_{i+1,j-1} + u_{i-1,j+1} + u_{i+1,j+1}}{6 h^2} - \frac{20 u_{i,j}}{6 h^2} \nonumber \\
    & \hspace{2cm} + \frac{4u_{i,j-1} + 4u_{i-1,j} + 4u_{i+1,j} +  4u_{i,j+1}}{6 h^2}. \label{fd_9}
\end{align}
For first-order partial derivatives on the Neumann boundary, we use the following forward or backward differences depending on the boundary location:
\begin{subequations} \label{fd_1st}
\begin{align}
    \partial_x^+ u(x_i,y_j) &= \frac{u_{i+1,j} - u_{i,j}}{h}, \qquad\partial_x^- u(x_i,y_j) = \frac{u_{i,j} - u_{i-1,j}}{h}, \\
    \partial_y^+ u(x_i,y_j) &= \frac{u_{i,j+1} - u_{i,j}}{h}, \qquad \partial_y^- u(x_i,y_j) = \frac{u_{i,j} - u_{i,j-1}}{h}.
\end{align}
\end{subequations}
For a boundary node $(x_i, y_j)$, we use $\partial_x^+ u$ if the interior neighbor in the x-direction is at index $i+1$; we use $\partial_x^- u$ if the interior neighbor is at index $i-1$.
The same rule applies in the $y$-direction.
With $\nabla_h u = (\partial_x^\pm u, \partial_y^\pm u)$, the finite difference method is to find $u_h$ such that
\begin{subequations}
\begin{align}
    -\Delta_h u_h(x_i, y_j) &= f(x_i, y_j) \qquad  \hskip+6pt \forall (x_i, y_j) \in \Omega\setminus\partial\Omega, \\
    u_h(x_i, y_j) &= g_D(x_i, y_j) \qquad \forall (x_i, y_j) \in \partial\Omega_D, \\
    \nabla_h u_h(x_i, y_j)\cdot\boldsymbol{n} &= g_N(x_i, y_j) \qquad \forall (x_i, y_j) \in \partial\Omega_N. \label{fd_neumann}
\end{align}
\end{subequations}

\section{FD-CON: Finite Difference Convolutional Operator Network}\label{subsec:FDUNet}

For completeness, we also present the finite difference convolutional operator network (FD-CON), which serves as a finite difference counterpart of FE-CON. Similar to FE-CON, FD-CON learns the solution operator by minimizing a physics-based residual, but uses a finite difference discretization instead of a finite element formulation.
The FD-CON approach approximates solutions of Poisson equations by minimizing a discretized strong-form residual.
In this formulation, the model is trained to satisfy both the PDE and boundary conditions simultaneously through a combination of residual-based loss terms.
The inputs to the network are provided as image channels. As shown in Figure~\ref{fig:fdm_input_channels}, the source term $f$ is encoded as a dense pixel-wise field over the entire domain $\Omega$. In contrast, the Dirichlet and Neumann boundary conditions, $g_D$ and $g_N$, are represented as separate input channels whose pixel values can be nonzero only on their respective boundary locations and are set to zero in the interior of the domain.
The network predicts the solution over the entire domain.
After training, the predicted solution is post-processed by replacing the values on the Dirichlet boundary with the given boundary condition, as illustrated in Figure~\ref{fig:fdm_input_channels}.
Therefore, the final solution satisfies the Dirichlet boundary condition exactly while retaining the network prediction in the interior and on Neumann boundaries.

\begin{figure}[t!]
\centering
\begin{subfigure}[t]{0.25\textwidth}
\centering
\begin{tikzpicture}[scale=0.4]
\foreach \x in {0,...,7} {
  \foreach \y in {0,...,7} {
    \fill[blue!35] (\x,\y) rectangle ++(1,1);
  }
}
\draw[step=1cm,black,thick] (0,0) grid (8,8);
\draw[black, line width = 1pt] (0,0)--(8,0)--(8,8)--(0,8)--(0,0)--(1,0);
\end{tikzpicture}
\caption{Source term $f$}
\label{fig:fdm_input_f}
\end{subfigure} \hspace{-.3cm}
\begin{subfigure}[t]{0.25\textwidth}
\centering
\begin{tikzpicture}[scale=0.4]
\foreach \x in {0,...,7} {
  \fill[red!35] (\x,0) rectangle ++(1,1);
  \fill[red!35] (\x,7) rectangle ++(1,1);
}
\foreach \y in {0,...,7} {
  \fill[red!35] (7,\y) rectangle ++(1,1);
}
\draw[step=1cm,black,thick] (0,0) grid (8,8);
\draw[black, line width = 1pt] (0,0)--(8,0)--(8,8)--(0,8)--(0,0)--(1,0);
\end{tikzpicture}
\caption{Dirichlet B.C. $g_D$}
\label{fig:fdm_input_gD}
\end{subfigure} \hspace{-.3cm}
\begin{subfigure}[t]{0.25\textwidth}
\centering
\begin{tikzpicture}[scale=0.4]
\foreach \y in {0,...,7} {
  \fill[green!35] (0,\y) rectangle ++(1,1);
}
\draw[step=1cm,black,thick] (0,0) grid (8,8);
\draw[black, line width = 1pt] (0,0)--(8,0)--(8,8)--(0,8)--(0,0)--(1,0);

\end{tikzpicture}
\caption{Neumann B.C. $g_N$}
\label{fig:fdm_input_gN}
\end{subfigure} \hspace{-.3cm}
\begin{subfigure}[t]{0.25\textwidth}
\centering
\begin{tikzpicture}[scale=0.4]

\foreach \x in {0,...,7} {
  \foreach \y in {0,...,7} {
    \fill[red!35] (\x,\y) rectangle ++(1,1);
  }
}

\foreach \x in {0,...,6} {
  \foreach \y in {1,...,6} {
    \fill[yellow!30] (\x,\y) rectangle ++(1,1);
  }
}

\draw[step=1cm,black,thick] (0,0) grid (8,8);
\draw[black, line width = 1pt] (0,0)--(8,0)--(8,8)--(0,8)--(0,0)--(1,0);

\end{tikzpicture}
\caption{Predicted solution $\hat{u}$}
\label{fig:fdm_output}
\end{subfigure}

\caption{Input channels (a)--(c) and the predicted solution (d) for FD-CON. The source term $f$ is provided over the entire domain, while the Dirichlet and Neumann boundary conditions are encoded as separate channels. The predicted solution is post-processed to enforce the Dirichlet boundary condition.}
\label{fig:fdm_input_channels}
\end{figure}

For computational efficiency, we compute the finite difference Laplacian using convolutional kernels.
The classical 5-point stencil \eqref{fd_5} and 9-point stencil \eqref{fd_9} for the discrete Laplacian can be expressed as 2D convolutions with fixed kernels.
For example, the discrete Laplacian with the 5-point stencil can be represented as follows:
\[
\Delta_h u_{i,j} = \underbrace{\frac{1}{h^2}}_{=: \alpha_5}
\underbrace{\begin{bmatrix}
0 & 1 & 0 \\
1 & -4 & 1 \\
0 & 1 & 0
\end{bmatrix}}_{=: K_5}
* 
\underbrace{\begin{bmatrix}
u_{i-1,j-1} & u_{i-1,j} & u_{i-1,j+1} \\
u_{i,j-1}   & u_{i,j}   & u_{i,j+1} \\
u_{i+1,j-1} & u_{i+1,j} & u_{i+1,j+1}
\end{bmatrix}}_{=: U(i,j)},
\]
where $U(i,j)$ denotes the $3\times 3$ local patch centered at $(i,j)$, and $*$ denotes the discrete convolution operator used in standard CNN implementations.
Similarly, for the 9-point stencil,
\[
\Delta_h u_{i,j} = \underbrace{\frac{1}{6h^2}}_{=: \alpha_9}
\underbrace{\begin{bmatrix}
1 & 4 & 1 \\
4 & -20 & 4 \\
1 & 4 & 1
\end{bmatrix}}_{=:K_9}
* 
\begin{bmatrix}
u_{i-1,j-1} & u_{i-1,j} & u_{i-1,j+1} \\
u_{i,j-1}   & u_{i,j}   & u_{i,j+1} \\
u_{i+1,j-1} & u_{i+1,j} & u_{i+1,j+1}
\end{bmatrix}.
\]
Then, the discrete Laplacian at $(i,j)$ can be written in the general form
\[
\Delta_h u_{i,j} = \alpha_s \, (K_s * U(i,j)), \qquad s\in\{5, \ 9\}
\]
where $s$ denotes the stencil size.
This formulation provides a unified view of finite difference discretizations and convolutional operators, which we exploit in our model design.

\begin{figure}[t!]
\centering

\begin{subfigure}[t]{0.3\textwidth}
\centering
\begin{tikzpicture}[scale=0.45]
\foreach \x in {1,...,6} {
  \foreach \y in {1,...,6} {
    \fill[blue!35] (\x,\y) rectangle ++(1,1);
  }
}
\draw[step=1cm,black,thick] (0,0) grid (8,8);
\draw[black, line width = 1pt] (0,0)--(8,0)--(8,8)--(0,8)--(0,0)--(1,0);
\end{tikzpicture}
\caption{$M_f$}
\label{fig:interior_masking}
\end{subfigure}
\hfill
\begin{subfigure}[t]{0.3\textwidth}
\centering
\begin{tikzpicture}[scale=0.45]
\foreach \x in {0,...,7} {
  \fill[red!35] (\x,0) rectangle ++(1,1);
  \fill[red!35] (\x,7) rectangle ++(1,1);
}
\foreach \y in {0,...,7} {
  \fill[red!35] (7,\y) rectangle ++(1,1);
}
\draw[step=1cm,black,thick] (0,0) grid (8,8);
\draw[black, line width = 1pt] (0,0)--(8,0)--(8,8)--(0,8)--(0,0)--(1,0);
\end{tikzpicture}
\caption{$M_{D}$}
\label{fig:diri_masking}
\end{subfigure}
\hfill
\begin{subfigure}[t]{0.3\textwidth}
\centering
\begin{tikzpicture}[scale=0.45]
\foreach \y in {0,...,7} {
  \fill[green!35] (0,\y) rectangle ++(1,1);
}
\draw[step=1cm,black,thick] (0,0) grid (8,8);
\draw[black, line width = 1pt] (0,0)--(8,0)--(8,8)--(0,8)--(0,0)--(1,0);

\end{tikzpicture}
\caption{$M_N$}
\label{fig:neu_masking}
\end{subfigure}

\caption{Masks for source term $M_f$, Dirichlet boundary $M_D$, and Neumann boundary $M_N$}
\label{fig:fdm_input_masks}
\end{figure}

To enforce the governing PDE over the interior of the domain, we define the interior loss by applying a convolution kernel \( K_s \) to the predicted solution \( \hat{u}^k \) and comparing it with the scaled source term \( f^k \). The loss is evaluated only over a masked region $M_f$, which excludes boundary pixels (see Figure~\ref{fig:interior_masking}):

\begin{equation} \label{eq:interior_loss}
\mathcal{L}_{f}^k = \frac{1}{|M_f|} \sum_{(i,j)\in M_f} \left| K_s * \hat{U}^k(i,j) + \frac{1}{\alpha_s} f_{i,j}^k \right|^2,
\end{equation}
where $|M_f|$ is the number of blue pixels in $M_f$ and $\hat{U}^k(i,j)$ denotes the $3\times3$ local patch extracted from the predicted solution $\hat{u}^k$ centered at $(i,j)$.
The scaling factor $\alpha_s$ increases quadratically as the grid spacing $h$ decreases. Incorporating such a large factor directly into the convolution kernel can lead to numerically unstable gradients and grid-dependent output scales. To mitigate this, we apply the scaling to the source term $f$ instead and use the unscaled kernel $K_s$ during convolution.

To enforce the Dirichlet boundary condition, we define the Dirichlet loss by comparing the predicted solution $\hat{u}$ with the prescribed boundary data $g_D$. The loss is computed only over a masked region $M_D$, corresponding to the Dirichlet boundary (see Figure~\ref{fig:diri_masking}):
\begin{equation} \label{eq:dirichlet_loss}
\mathcal{L}_{D}^k = \frac{1}{|M_D|} \sum_{(i,j)\in M_D} \left| \hat{u}_{i,j}^k - g_{Di,j}^k \right|^2,
\end{equation}
where $|M_D|$ is the number of red pixels in $M_D$.

Since we assume that the Neumann boundary is the left boundary of the domain, the discrete normal derivative is approximated by a first-order finite difference:

\[
\nabla_h u(x_i,y_j)\cdot\boldsymbol{n} = -\frac{u_{i+1,j} - u_{i,j}}{h}.
\]
Then the loss for the Neumann boundary is given by:
\begin{equation} \label{eq:neumann_loss}
\mathcal{L}_{N}^k = \frac{1}{|M_N|} \sum_{(i,j)\in M_N} \left| \hat{u}_{i+1, j}^k - \hat{u}_{i,j}^k + h\,g_{Ni,j}^k \right|^2,
\end{equation}
where $|M_N|$ is the number of green pixels in $M_N$ (see Figure~\ref{fig:neu_masking}).
As with the interior loss, we avoid including the $1/h$ factor inside the difference and instead scale $g_N$ by $h$ to ensure consistent loss magnitudes across different resolutions.

The total loss is a weighted combination of the three terms:
\begin{equation}\label{fd_loss}
    \mathcal{L}_{\text{FDM}} = \frac{1}{N_s} \sum_{k=1}^{N_s}\left( \lambda_{f} \mathcal{L}_{f}^k + \lambda_{D} \mathcal{L}_{D}^k + \lambda_{N} \mathcal{L}_{N}^k\right),
\end{equation}
where $\lambda_{f}, \lambda_{D}$ and $\lambda_{N}$ are hyperparameters, and $N_s$ is the number of samples.

In order to train the models for the subproblems \eqref{Subproblem1} and~\eqref{Subproblem2},
each model is trained using only the input channels corresponding to nonzero components.
More specifically, for Subproblem 1, a two-channel input corresponding to $f$ and $g_N$ (see Figures~\ref{fig:fdm_input_f} and \ref{fig:fdm_input_gN}) is used, while $g_D = 0$ is enforced in \eqref{eq:dirichlet_loss}.
For Subproblem 2, a single-channel input corresponding to $g_D$ (see Figure~\ref{fig:fdm_input_gD}) is used, while $f = 0$ and $g_N = 0$ are enforced in \eqref{eq:interior_loss} and~\eqref{eq:neumann_loss}, respectively.
When a model predicts the solution, the values on the Dirichlet boundary are post-processed by overwriting with the exact Dirichlet data (see Figure~\ref{fig:fdm_output}). The boundary values are set using $g_D$ for both the original problem and Subproblem 2, while the values are set to zero for Subproblem 1.

\subsection{Error Analysis}

Let $\mathcal{V}_I$ denote the set of interior grid nodes and $\mathcal{V}_B$ the set of boundary grid nodes.
We reorder the nodes so that interior nodes come first, followed by boundary nodes.
Accordingly, a grid function is written as
\begin{equation}\label{grid_fn}
\boldsymbol{v} = \left[\begin{array}{c} \boldsymbol{v}_I \\ \boldsymbol{v}_B\end{array}\right] \in \mathbb{R}^{|\mathcal{V}_I|+|\mathcal{V}_B|},
\end{equation}
where $\boldsymbol{v}_I \in \mathbb{R}^{|\mathcal{V}_I|}$, $\boldsymbol{v}_B \in \mathbb{R}^{|\mathcal{V}_B|}$.
Let $A_I \in \mathbb{R}^{|\mathcal{V}_I|\times |\mathcal{V}_I|}$ denote the standard finite difference matrix acting on interior nodes only, obtained from any standard symmetric stencil (e.g., 5-point or 9-point) with homogeneous Dirichlet boundary conditions enforced by elimination.
Then $A_I$ is symmetric positive definite.
We use the standard scaling $A_I\sim \alpha_s \widetilde{A}_I$, where $\widetilde{A}_I$ is independent of $h$.
Under this convention, the smallest eigenvalue of $A_I$ is bounded away from zero uniformly in $h$, while the largest eigenvalue grows like $O(h^{-2})$.

Let us define the (weighted) discrete $L^2$ inner product and norm on interior node vectors by
\begin{equation}\label{discrete_l2}
    (\boldsymbol{v}_I, \boldsymbol{w}_I)_h := h^2 \boldsymbol{v}_I^T \boldsymbol{w}_I, \qquad \|\boldsymbol{v}_I\|^2_{0,h} := (\boldsymbol{v}_I, \boldsymbol{v}_I)_h.
\end{equation}
The discrete $H^1$-seminorm is defined as
\begin{equation}\label{discrete_H1}
    |\boldsymbol{v}_I|^2_{1,h} := h^2 \sum_{(i,j) \in \mathcal{V}_I} |\nabla_h v_{i,j}|^2.
\end{equation}
For homogeneous Dirichlet boundary conditions, this seminorm is computed using the interior values with boundary values understood as zero.
It is well known that the associated discrete bilinear form is equivalent to this seminorm:
$$|\boldsymbol{v}|_{1,h}^2 \simeq h^2 \boldsymbol{v}_I^T A_I \boldsymbol{v}_I.$$
More precisely, there exist positive constants $c_I$ and $C_I$, independent of $h$, such that for all interior vectors $\boldsymbol{v}_I$,
\begin{equation}\label{discrete_norm_eq}
    c_I |\boldsymbol{v}_I|_{1,h}^2 \le h^2 \boldsymbol{v}_I^T A_I \boldsymbol{v}_I \le C_I |\boldsymbol{v}_I|_{1,h}^2.
\end{equation}
For a single input, the FD-CON loss functions \eqref{eq:interior_loss} and \eqref{eq:dirichlet_loss} can be written as
\begin{equation}\label{fd_loss_err}
    \mathcal{L}_f = \frac{1}{|\mathcal{V}_I|} \left\|\widetilde{A}_I \hat{\boldsymbol{u}}_I - \frac{1}{\alpha_s}\boldsymbol{f}_I\right\|^2_2, \qquad \mathcal{L}_D = \frac{1}{|\mathcal{V}_B|} \|\hat{\boldsymbol{u}}_B\|^2_2,
\end{equation}
where $\widetilde{A}_I = \alpha_s^{-1} A_I$, $\boldsymbol{f}_I\in \mathbb{R}^{|\mathcal{V}_I|}$ and $f_j = f(\eta_j), \ \forall\eta_j \in \mathcal{V}_I$. Then, we have the following error estimate.

\begin{theorem}
    Let $u \in H^2(\Omega) \cap H^1_D(\Omega)$ be the weak solution, and let $\hat{\boldsymbol{u}}$ be the FD-CON prediction.
    Then there exist positive constants $C_1$, $C_2$, and $C_3$, independent of $h$, such that
    \begin{equation}
        |\boldsymbol{u}_I - \hat{\boldsymbol{u}}_I|_{1,h}^2 + h^{-1}\|\boldsymbol{u}_B - \hat{\boldsymbol{u}}_B\|^2_{0,h} \le C_1 h^2 \|u\|^2_{H^2(\Omega)} + C_2 h^{-4}\mathcal{L}_f + C_3 \mathcal{L}_D,
    \end{equation}
    where $\boldsymbol{u} = [\boldsymbol{u}_I^T, \boldsymbol{u}_B^T]^T$ is the reordered grid function \eqref{grid_fn} corresponding to the weak solution, i.e., $u_j = u(\eta_j), \ \eta_j \in \mathcal{V}_I \cup \mathcal{V}_B, \ j = 1, ..., N^2$.
\end{theorem}
\begin{proof}
Clearly, $\boldsymbol{u}_B = \boldsymbol{0}$. 
Let $\boldsymbol{u}^*_h$ be the finite difference solution, satisfying
$$\boldsymbol{u}^*_h = \left[\begin{array}{c} \boldsymbol{u}^h_I \\ \boldsymbol{0}\end{array}\right] \in \mathbb{R}^{|\mathcal{V}_I|+|\mathcal{V}_B|} \qquad \textrm{and}\qquad A_I \boldsymbol{u}^h_I = \boldsymbol{f}_I.$$
By the triangle inequality, the discrete norm equivalence \eqref{discrete_norm_eq}, and the standard finite difference error estimates, we have
\begin{align*}
     |\boldsymbol{u}_I - \hat{\boldsymbol{u}}_I|_{1,h}^2 &\le 2 |\boldsymbol{u}_I - \boldsymbol{u}^h_I|_{1,h}^2 + 2 |\boldsymbol{u}^h_I - \hat{\boldsymbol{u}}_I|_{1,h}^2 \le C h^2\|u\|^2_{H^2(\Omega)} + 2 |\boldsymbol{u}^h_I - \hat{\boldsymbol{u}}_I|_{1,h}^2.
\end{align*}
Cauchy-Schwarz inequality and the discrete Poincar\'e inequality imply
\begin{align*}
    |\boldsymbol{u}^h_I - \hat{\boldsymbol{u}}_I|_{1,h}^2 &\le c_I^{-1} h^2 (\boldsymbol{u}^h_I - \hat{\boldsymbol{u}}_I)^T A_I (\boldsymbol{u}^h_I - \hat{\boldsymbol{u}}_I) \\
    &\le c_I^{-1} h \|\boldsymbol{u}^h_I - \hat{\boldsymbol{u}}_I\|_{0,h} \|\boldsymbol{f}_I - A_I\hat{\boldsymbol{u}}_I\|_2 \\
    &\le \alpha_s c_I^{-1} C_P h |\boldsymbol{u}^h_I - \hat{\boldsymbol{u}}_I|_{1,h} \left\|\widetilde{A}_I\hat{\boldsymbol{u}}_I -\frac{1}{\alpha_s}\boldsymbol{f}_I\right\|_2.
\end{align*}
Since $\alpha_s \sim h^{-2}$ , we have
$$|\boldsymbol{u}^h_I - \hat{\boldsymbol{u}}_I|_{1,h} \le C h^{-1} \left\|\widetilde{A}_I\hat{\boldsymbol{u}}_I -\frac{1}{\alpha_s}\boldsymbol{f}_I\right\|_2.$$
Combining the above results with $|\mathcal{V}_I| \sim h^{-2}$ yields
\begin{equation}\label{fd_err_1}
    |\boldsymbol{u}_I - \hat{\boldsymbol{u}}_I|_{1,h}^2 \le C h^2\|u\|^2_{H^2(\Omega)} + 2 C h^{-2}\left\|\widetilde{A}_I\hat{\boldsymbol{u}}_I -\frac{1}{\alpha_s}\boldsymbol{f}_I\right\|_2^2 \le C h^2\|u\|^2_{H^2(\Omega)} + 2 C h^{-4} \mathcal{L}_f.
\end{equation}
By \eqref{discrete_l2} and the fact that $|\mathcal{V}_B| \sim h^{-1}$, we have
\begin{equation}\label{fd_err_2}
    h^{-1}\|\boldsymbol{u}_B - \hat{\boldsymbol{u}}_B\|^2_{0,h} \le C  \frac{1}{|\mathcal{V}_B|} \|\hat{\boldsymbol{u}}_B\|^2_2 = C \mathcal{L}_D.
\end{equation}
Combining \eqref{fd_err_1} and \eqref{fd_err_2} concludes the proof.
\end{proof}

\begin{corollary} \label{cor_train_fd}
Let $u^k\in H^2(\Omega)\cap H^1_D(\Omega)$ be the weak solution, and let $\hat{\boldsymbol{u}}^k$ be the FD-CON prediction corresponding to $f^k$in the homogeneous Dirichlet boundary setting considered above. 
Then
\begin{equation}
    \frac{1}{N_s} \sum_{k=1}^{N_s} \left(|\boldsymbol{u}^k_I - \hat{\boldsymbol{u}}_I^k|_{1,h}^2 + h^{-1
}\|\boldsymbol{u}_B^k - \hat{\boldsymbol{u}}_B^k\|^2_{0,h}\right) \le C_1 h^2 \frac{1}{N_s} \sum_{k=1}^{N_s} \|u^k\|_{H^2(\Omega)}^2 + C h^{-4} \mathcal{L}_{\text{FDM}}.
\end{equation}
\end{corollary}
\begin{remark} \label{rem:loss_scaling_fd}
    Corollary~\ref{cor_train_fd} shows that the averaged error consists of an $O(h^2)$ discretization term and a loss-dependent term scaled by $h^{-4}$.
    Hence, to attain the optimal convergence rate $O(h^2)$, the training loss $\mathcal{L}_{\textrm{FDM}}$ must be bounded by $O(h^6)$.
    For example, halving the mesh size requires reducing $\mathcal{L}_{\textrm{FDM}}$ by a factor of 64.
\end{remark}

\subsection{Numerical experiments}
\begin{figure}
    \centering
    \includegraphics[height=9cm]{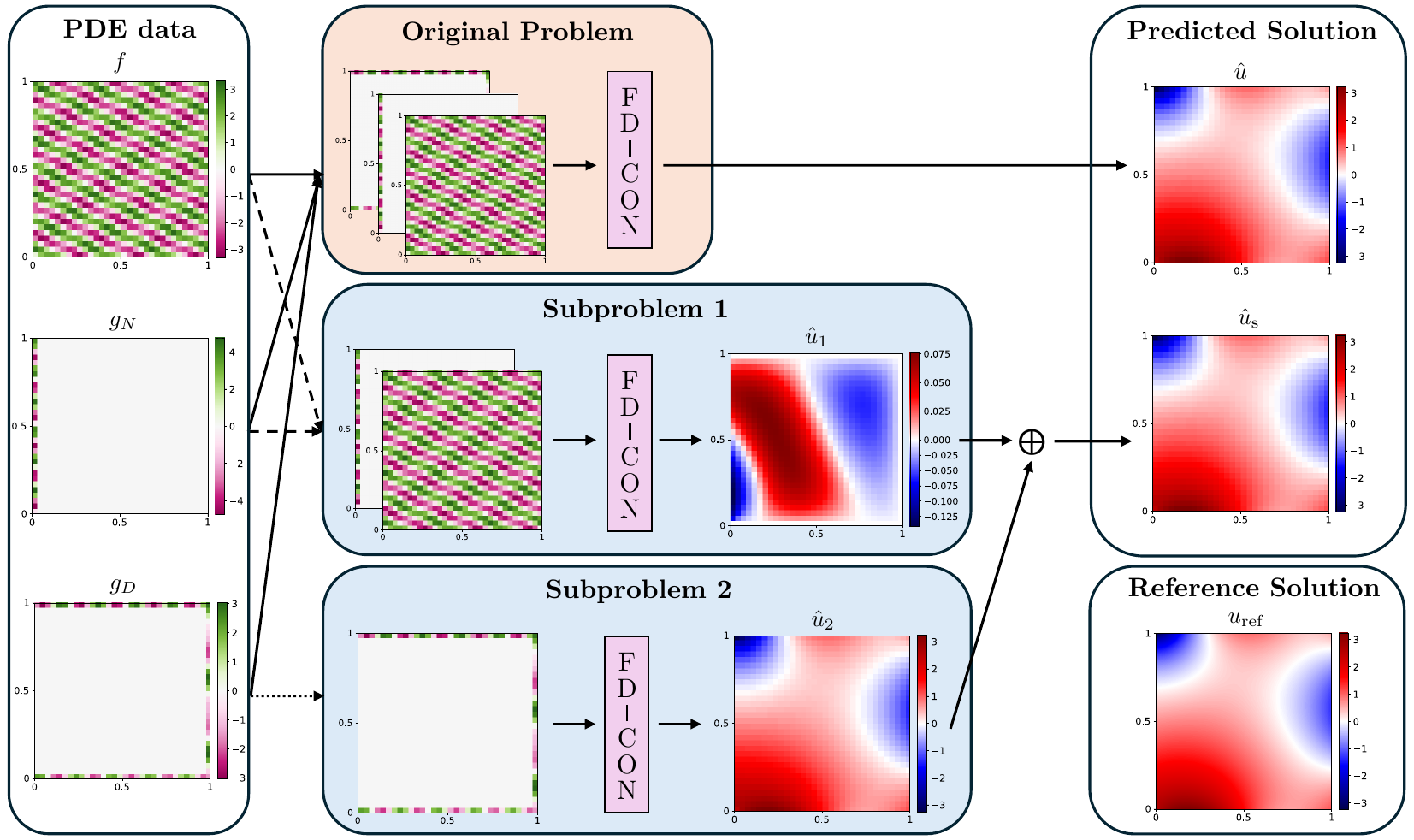}
    \caption{Computational workflow of FD-CON for $N=32$.}
    \label{fig:workflow_FD}
\end{figure}

Figure~\ref{fig:workflow_FD} shows the computational workflow of FD-CON.
The PDE data $f$, $g_N$, and $g_D$ are first transformed by the domain-to-image mapping, so that each input image encodes not only the values of the data but also the geometry of the domain and the boundary information.
For FD-CON, these images are used directly as network inputs.
For the original problem, the model takes a three-channel input that consists of $f$, $g_N$, and $g_D$.
For the decomposed formulation, two separate models are used: the model for Subproblem 1 uses a two-channel input corresponding to $f$ and $g_N$, while the model for Subproblem 2 uses a single-channel input corresponding to $g_D$.
Each model produces a corresponding prediction, $\hat{u}$, $\hat{u}_1$, and $\hat{u}_2$, respectively, and the final prediction is obtained by $\hat{u}_s = \hat{u}_1 + \hat{u}_2$.

Additional numerical results corresponding to Sections~\ref{ex:performance} and \ref{numerical_verification} are provided for FD-CON.
The observed behaviors in Figures~\ref{fig:app1} and \ref{fig:error_analysis_FD} are consistent with those of FE-CON, including convergence under grid refinement and the effect of loss scaling.

\begin{figure}[t!]
    \centering
    \begin{subfigure}{0.49\linewidth}
        \centering
        \includegraphics[height=6cm]{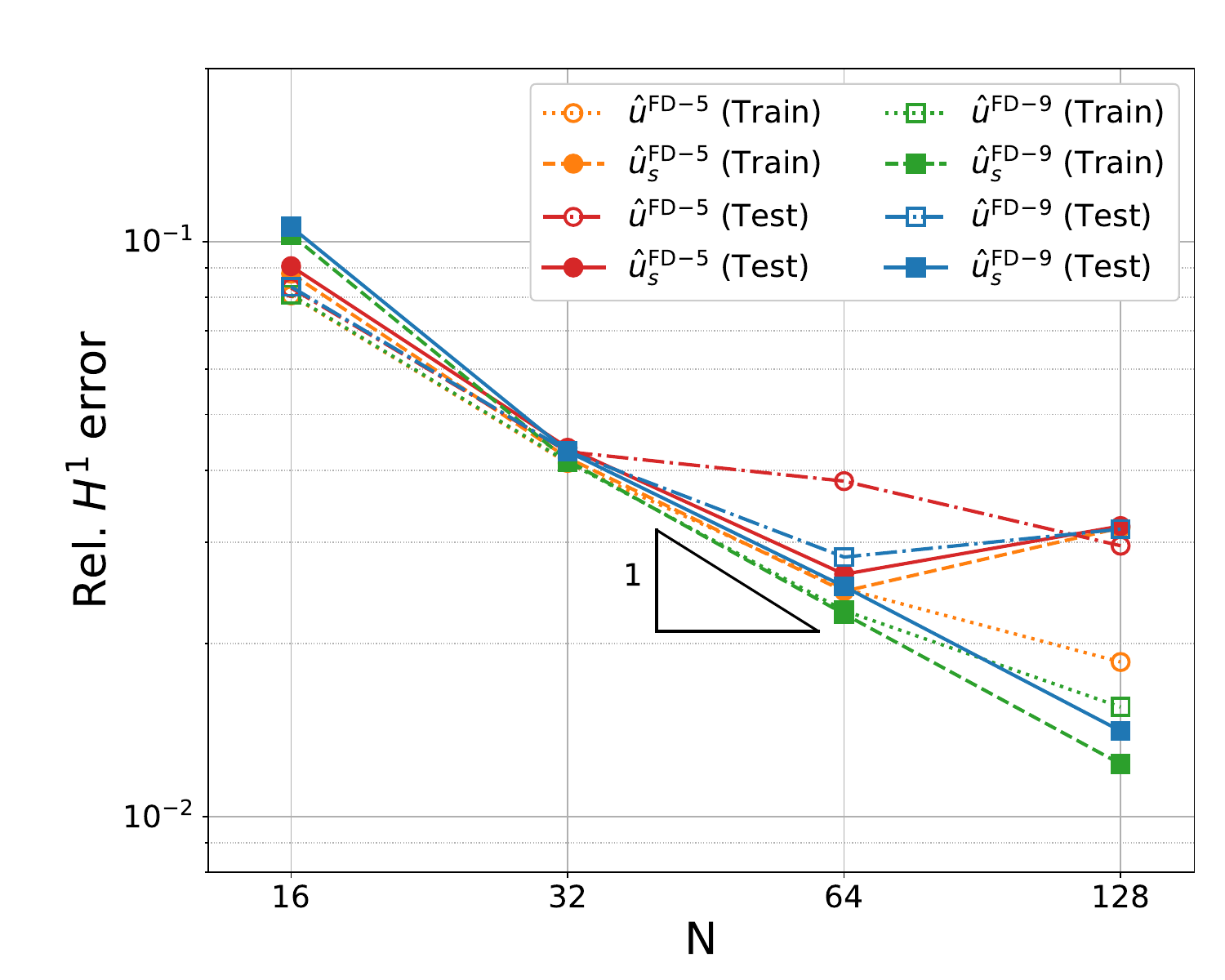}
        \caption{Relative error}
        \label{fig:app1_error_FD}
    \end{subfigure}
    \hfill
    \begin{subfigure}{0.49\linewidth}
        \centering
        \includegraphics[height=6cm]{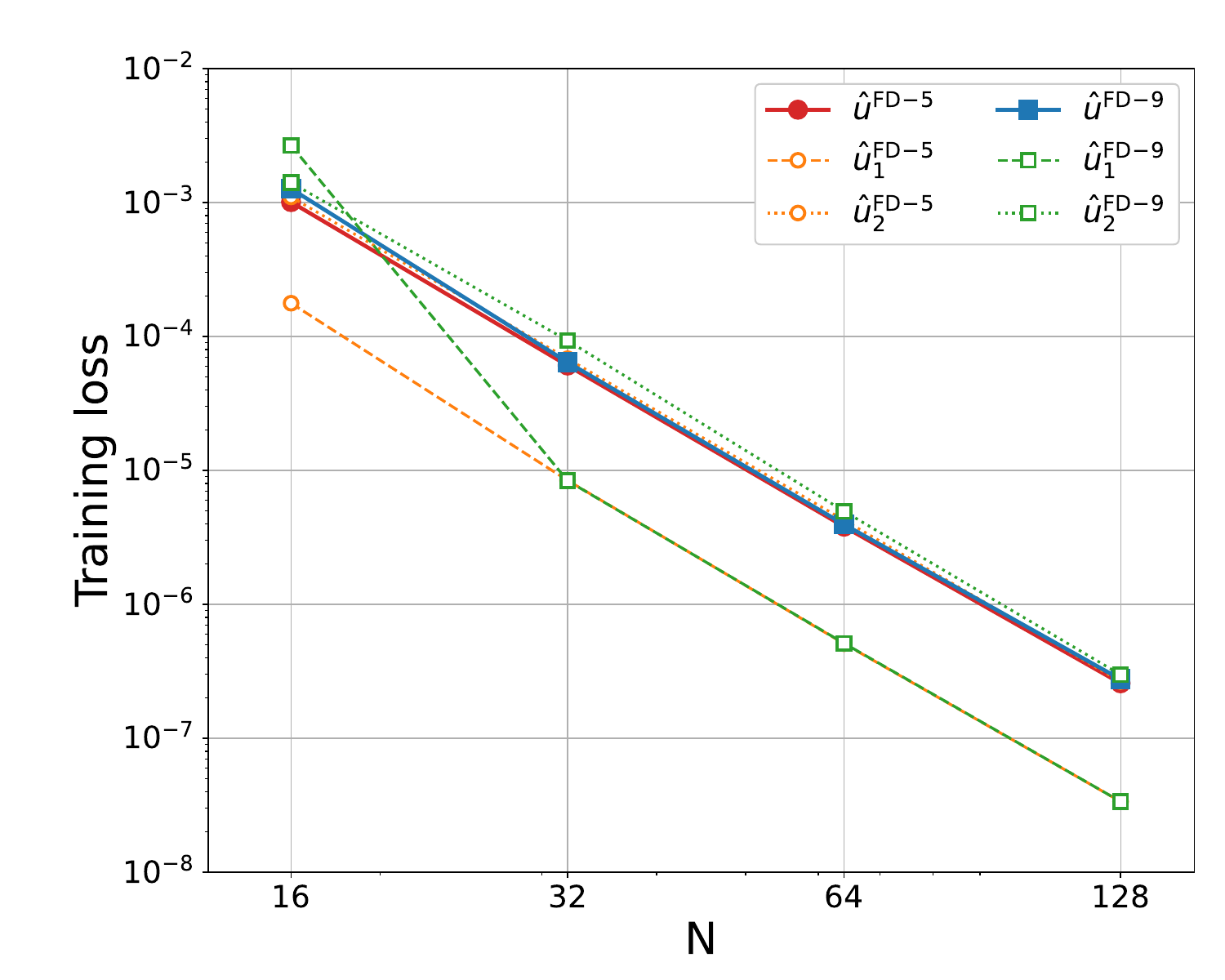}
        \caption{Training loss} 
        \label{fig:app1_loss_FD}
    \end{subfigure}
    \caption{Relative $H^1$ errors and training losses of FD-CON under grid refinement.
    }
    \label{fig:app1}
\end{figure}

\begin{figure}[t!]
    \centering
    \begin{subfigure}{0.49\linewidth}
        \centering
        \includegraphics[height=6cm]{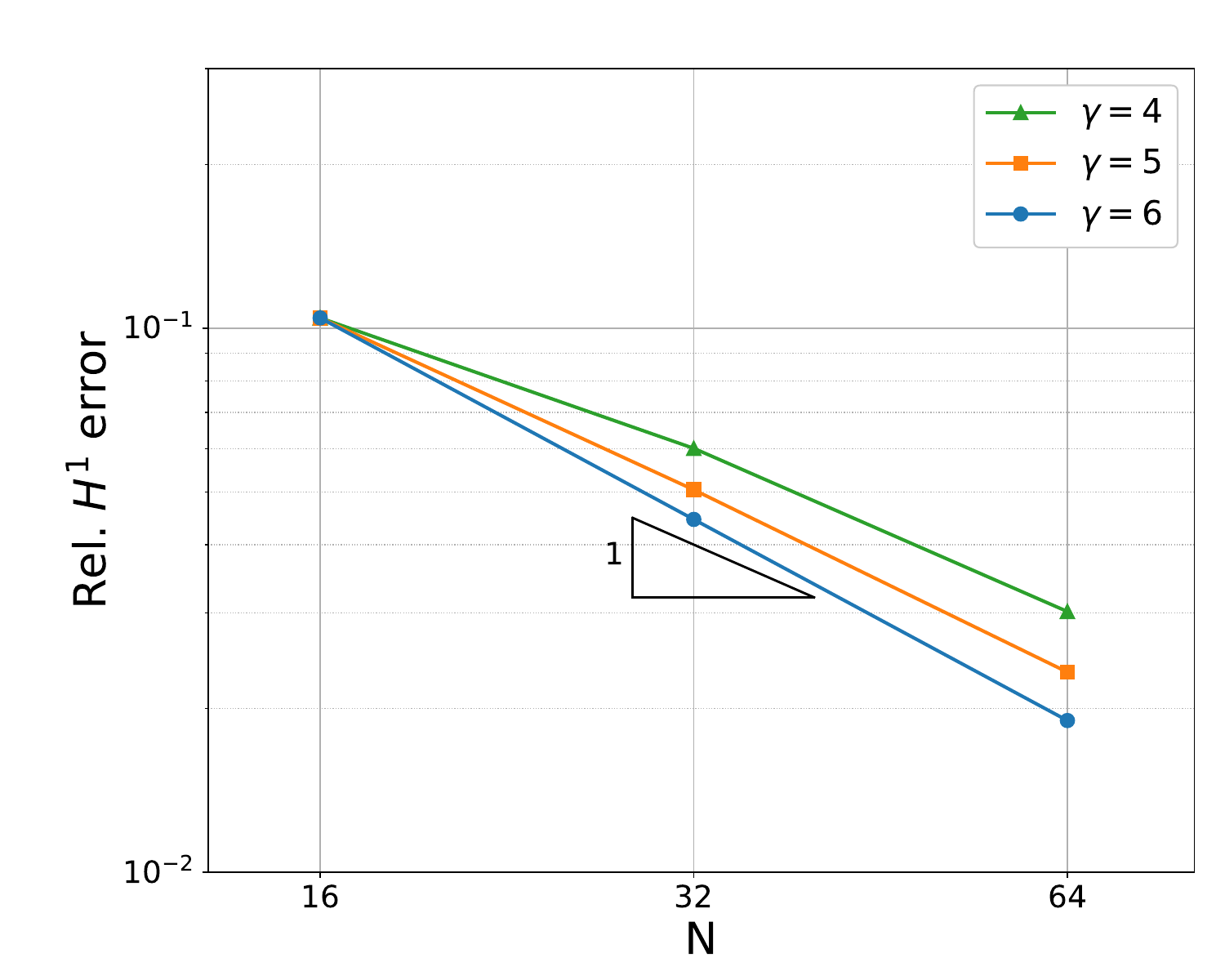}
        \caption{Training error} \label{fig:FD_9_error_analysis_train}
    \end{subfigure}
    \hfill
    \begin{subfigure}{0.49\linewidth}
        \centering
        \includegraphics[height=6cm]{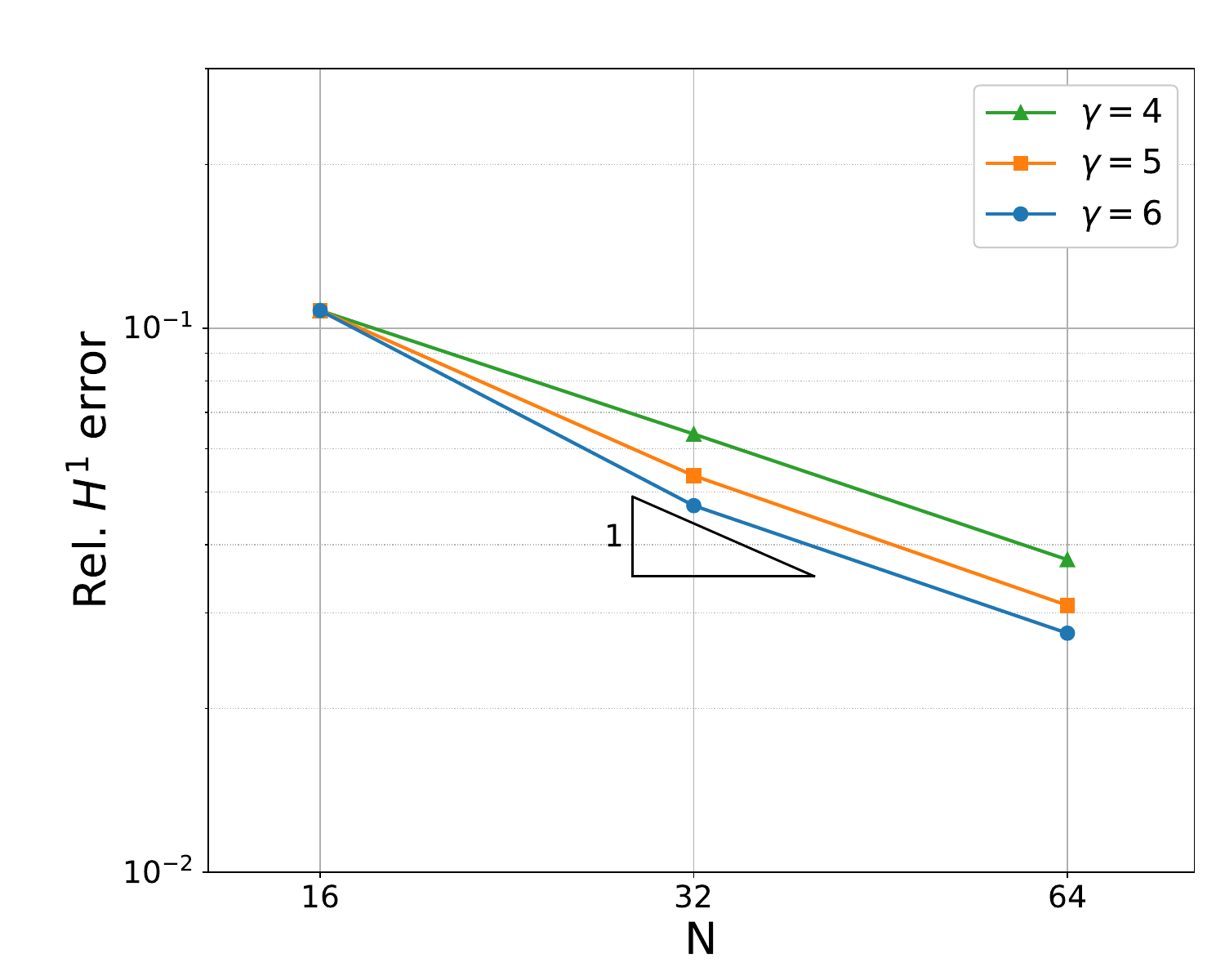}
        \caption{Test error} \label{fig:FD_9_error_analysis_test}
    \end{subfigure}
    \caption{Relative $H^1$ errors of FD-CON (FD-9) under different loss scalings $(O(h^\gamma))$.}
    \label{fig:error_analysis_FD}
\end{figure}

\end{document}